\documentclass[10pt,twoside,a4paper,english,reqno]{amsart}
\usepackage{amscd}
\usepackage{caption}
\usepackage{a4wide}
\usepackage{multirow}
\usepackage{booktabs}

\usepackage{amssymb,amsthm,amsmath,amsfonts,mathrsfs}
\usepackage{bm,latexsym,enumitem,dsfont,tabularx,graphicx,url}
\usepackage[utf8]{inputenc}
\usepackage{hyperref}
\usepackage{stmaryrd}
\hypersetup{
    colorlinks=true,
    linkcolor=blue,
    filecolor=red,
    urlcolor=blue,
    citecolor=red,
}
\usepackage{nicefrac,microtype,upref,ulem,doi,footmisc}

\theoremstyle{plain}
\newtheorem{theorem}{Theorem}[section]

\newtheorem{lemma}[theorem]{Lemma}

\newtheorem{remark}[theorem]{Remark}
\newtheorem{assumption}[theorem]{Assumption}

\numberwithin{equation}{section}

\usepackage{tikz,tikzscale}
\usepackage{pgfplots}
\pgfplotsset{compat=1.13}

\usepackage{xcolor}
\definecolor{lateksii_color}{RGB}{155,0,119}

\title[HDG for the Ostrovsky equation]{A hybridizable discontinuous Galerkin method for the Ostrovsky equation}
\author{Mukul Dwivedi \and Andreas Rupp}
\address{Department of Mathematics, Saarland University, Saarbrücken, Germany}
\email{\{mukul.dwivedi;andreas.rupp\}@uni-saarland.de}
\thanks{}


\begin{document}

\begin{abstract}
This paper develops the hybridizable discontinuous Galerkin (HDG) method for the Ostrovsky equation, a nonlinear dispersive wave equation featuring both third-order dispersion and a nonlocal antiderivative term with Coriolis effect. On a bounded interval, the nonlocal operator $\partial_x^{-1}$ is localized through an auxiliary variable $v$ satisfying $v_x=u$ together with an additional boundary constraint that ensures uniqueness.
We employ a mixed first-order formulation to decompose the dispersive operator and to localize the nonlocal term, and we couple the resulting semi-discrete HDG scheme with a $\theta$-time stepping method for $\theta \in [1/2,1]$.  We prove $L^2$-stability for suitable stabilization parameters and derive an {\it a priori} $L^2(\Omega)$ error estimate for smooth solutions that explicitly accounts for the nonlinear convective flux.
 Numerical examples illustrate the convergence properties and demonstrate the scheme's capability to handle smooth and non-smooth solutions, including solitary wave propagation and peaked solitary wave (peakon) propagation in the zero dispersive limit regime. 
\end{abstract}

\maketitle

\section{Introduction}

We study the Ostrovsky equation (also known as the rotation-modified Korteweg--de Vries (KdV) equation) 
\begin{equation}\label{eq:ostrovsky}
\begin{cases}
    \partial_t u - \beta \partial_x^3 u  + f(u)_x - \gamma \partial_x^{-1} u = 0, \quad &x \in \mathbb{R},\; t>0,\\
    u(x,0)=u_0(x), \qquad &x\in \mathbb R,
\end{cases}
\end{equation}
where $f(u) = \frac{\alpha}{2} u^2$, $u$ is a real valued function, $\beta \in \mathbb R$, $\beta\neq 0$, $\gamma>0$ and $\alpha>0$, and the nonlocal antiderivative term \(\partial_x^{-1}\) can be defined by
\(
(\partial_x^{-1} g)(x):=\int_{+\infty}^{x} g(s)\,\mathrm ds,
\)
for functions $g$ that decay as $x\to +\infty$. 
The Ostrovsky equation \eqref{eq:ostrovsky} was originally derived as an asymptotic model for weakly nonlinear long surface waves in a rotating frame of reference, where the Coriolis effect introduces a nonlocal restoring term. It was first proposed by Ostrovsky \cite{ostrovsky1978nonlinear} in the context of oceanic wave propagation; see also \cite{redekopp1983nonlinear} for additional derivations and physical discussion. The equation captures the balance between nonlinear steepening, high-frequency dispersion (when \(\beta \neq 0\)), and the nonlocal effect of background rotation through the term \(\gamma \partial_x^{-1}u\). Here \(\gamma\) measures the strength of the Coriolis effect \cite{ostrovsky1978nonlinear}, while \(\beta\) determines the type of dispersion \cite{gui2007cauchy,levandosky2006stability,linares2006local,liu2007stability,yan2018cauchy}. This nonlocal operator introduces significant mathematical and computational challenges, distinguishing \eqref{eq:ostrovsky} from purely local dispersive models such as the KdV equation \cite{kenig1991well}, which is recovered when \(\gamma = 0\). In the limit \(\beta \to 0\), the equation reduces to the Ostrovsky--Hunter (OH) or Vakhnenko equation \cite{coclite2014convergence,coclite2015dispersive,coclite2020solutions}, which models short-wave perturbations in relaxing media and bubbly liquids and exhibits wave breaking and singular solitons such as peakons, cuspons, and loop solitons \cite{grimshaw2012reduced, liu2010wave,vakhnenko1992solitons,vakhnenko1998two}.

The existence and stability of solitary wave solutions for the Ostrovsky equation \eqref{eq:ostrovsky} has been an interesting topic of research, motivated by its structural similarities with the integrable KdV and Kadomtsev–Petviashvili (KP) equations \cite{KadomtsevPetviashvili1970}. Extensive analytical and numerical investigations have been devoted to this problem in  \cite{chen2001analytical}. For the full dispersion case (\(\beta \neq 0\)), the existence and stability properties depend critically on both parameters \(\beta\) and \(\gamma\). When \(\beta > 0\), solitary waves exist for phase speeds $c_w$ satisfying \(c_w < 2\sqrt{\gamma\beta}\), with the set of ground states being stable for certain values of \(c_w\) \cite{liu2004stability}. Conversely, for \(\beta < 0\), no solitary wave solutions exist when the phase speed $c_w$ satisfies \(c_w < \sqrt{140\gamma|\beta|}\) \cite{liu2004stability}. Since explicit closed-form expressions for these solitary waves are generally unavailable, numerical methods become essential for their approximation and for studying dynamical behaviors such as the persistence of a perturbed KdV soliton under the influence of weak rotation (small \(\gamma\)).

Analytically, the Cauchy problem for \eqref{eq:ostrovsky} has been extensively studied. Local and global well-posedness in the anisotropic Sobolev spaces \(X_s = \{ f \in H^s(\mathbb{R}) : \partial_x^{-1}f \in L^2(\mathbb{R}) \}\) for \(s > \frac{3}{4}\) were established by Linares and Milanes \cite{linares2006local} using energy methods and dispersive estimates. More recent works have extended these results to lower regularity and provided refined stability analyses for solitary waves \cite{coclite2020solutions,gui2007cauchy,wang2018cauchy, yan2018cauchy,yan2019white}. The existence and stability of solitary waves crucially depend on the sign of \(\beta\gamma\): for \(\beta\gamma > 0\), solitary waves exist and are stable, while for \(\beta\gamma < 0\) they may be unstable or non-existent \cite{levandosky2006stability, levandosky2007stability, liu2007stability, liu2004stability}. Very recently, the spectral stability of constrained solitary waves for a generalized Ostrovsky equation, with $\beta>0$ and $\gamma<0$, has been investigated in \cite{han2024spectral}, providing new insights into the stability properties under additional constraints. Equation \eqref{eq:ostrovsky} also possesses at least three conserved quantities \cite{gui2007cauchy}, namely
\begin{align}\label{energy}
  E(u)&= \int_{\mathbb R} u^2 \,dx  ,\\
   V(u)&=  \int_{\mathbb R} \big(\frac{1}{3}u^3 + \frac{\gamma}{2}(\partial_x^{-1}u)^2 + \beta(\partial_xu)^2 \big)dx,\\
   I(u)&= \int_{\mathbb R} u \,dx.
\end{align}
The limit \(\beta \to 0\) (convergence to the OH equation) has been rigorously justified in \cite{levandosky2007stability, liu2010wave}, revealing a rich structure of singular traveling waves and wave breaking phenomena, and convergence of solutions in the limit ($\beta \to 0$ and $\gamma\to 0$) to the Burgers equation in \cite{coclite2015dispersive}.

Numerical approximation of the full Ostrovsky equation ($\beta \neq 0$) is particularly challenging due to the interplay of third-order dispersion, nonlinearity, and the nonlocal antiderivative term.  Standard finite difference, finite element, and spectral methods often struggle to handle both the high-order derivative and the nonlocal operator simultaneously without inducing spurious oscillations or loss of stability. As noted by Kawai et al. \cite{kawai2025mathematical}, numerical methods for the full equation are still under development, and comprehensive mathematical analyses remain scarce. For the reduced Ostrovsky--Hunter equation ($\beta = 0$), several specialized schemes have been proposed, including finite difference methods \cite{coclite2017convergent, ridder2019convergent} and discontinuous Galerkin (DG) methods \cite{zhang2020discontinuous}. For the KdV equation ($\gamma = 0$), various numerical methods with extensive analysis has been studied as well, see \cite{Dong2017,dwivedi2025convergence,dwivedi2026analysis,holden1999operator,Yan2011} and references therein. 
Some numerical schemes for the Ostrovsky equation \eqref{eq:ostrovsky} have been developed in recent years. First finite difference scheme proposed in \cite{hunter1990dispersive}, however, the complete detail was not described to investigate solutions under both positive and negative dispersion effects. In \cite{grimshaw1998terminal}, Fourier–Galerkin scheme and in \cite{chen2001analytical}, Fourier-pseudo-spectral scheme are used to examine the evolutions of soliton-like solutions in a periodic setting. The numerical integration of \eqref{eq:ostrovsky} based on its geometric structures given in \cite{MiyatakeYaguchiMatsuo2012}. Recent efforts for the full equation on a periodic domain include some conservative numerical schemes \cite{yaguchi2010conservative}, a norm-conservative finite difference scheme with  error estimates  \cite{kawai2025mathematical}, and a geometric integration approach via the scalar auxiliary variable (SAV) method \cite{sato2025geometric}, both providing structure-preserving properties. The DG framework, with its flexibility in handling higher-order derivatives and nonstandard operators, offers a promising alternative. However, traditional DG methods typically result in a large number of globally coupled degrees of freedom \cite{CockburnGopalakrishnanLazarov2009}, which motivates the development of more efficient alternatives such as the hybridizable discontinuous Galerkin (HDG) method.

The hybridizable discontinuous Galerkin (HDG) method, introduced by Cockburn et al. \cite{CockburnGopalakrishnanLazarov2009}, addresses this issue by introducing hybrid variables on the mesh skeleton, allowing for local elimination of interior unknowns and significant reduction of the global system size. Since its inception, the HDG method has been successfully applied to a wide range of problems, including compressible flows \cite{PeraireNguyenCockburn2010}, linear elasticity \cite{WangHuangQiu2019}, and KdV type equations \cite{ChenCockburnDong2016,ChenDongJiang2018,Dong2017,Samii2016}. The method has also been extended to Maxwell's equations \cite{chen2017superconvergent}, nonlinear diffusion with internal jumps \cite{Musch2023hdg}, the non--local Camassa--Holm--Kadomtsev--Petviashvili equation \cite{dwivedi2026hybridizable}, and the Stokes equation with efficient multigrid solvers \cite{LuWangKanschatRupp2024}. Recent advancements include reduced stabilization techniques \cite{oikawa2015hybridized} and localized orthogonal decomposition strategies \cite{LuMaierRupp2025}, further enhancing the method's efficiency and applicability. Despite these developments, an HDG method for the full Ostrovsky equation \eqref{eq:ostrovsky} has not been previously proposed. We adopt an HDG discretization because it preserves the locality and conservation structure of DG methods
while reducing globally coupled degrees of freedom to trace unknowns on the mesh skeleton via static condensation,
which is particularly attractive for dispersive systems with multiple auxiliary variables \cite{ChenCockburnDong2016,CockburnGopalakrishnanLazarov2009}.

In this work, we develop the first high-order HDG method for the Ostrovsky equation \eqref{eq:ostrovsky} on a bounded interval \(\Omega = (x_L, x_R)\) with appropriate boundary conditions to ensure the well posedness. We decompose the Ostrovsky equation \eqref{eq:ostrovsky} into the first order system by introducing auxiliary variables.  The nonlocal term is localized through the auxiliary variable \(v\) by setting $v_x = u$ so that $v_h$ approximates $\partial_x^{-1}u$, which also satisfies an additional ordinary differential equation and an additional boundary constraint is given for the uniqueness of the operator $\partial_x^{-1}$ in a bounded domain. The signs of \(\beta\) and \(\gamma\) play a crucial role in the design of the numerical traces and the imposition of boundary conditions. 

Our HDG formulation employs polynomial spaces of degree \(k\) on a one-dimensional mesh and introduces carefully designed numerical traces on element interfaces with some stabilization parameters independent of $h$, the mesh discretization parameter. The global HDG scheme is obtained by assembling the local formulations and imposing transmission conditions weakly. The method is shown stable and error estimates are obtained for the nonlinear flux. In particular, we derive an error estimate of order \(O(h^{k+1/2})\) in the \(L^2\)-norm for the smooth solution \(u\). This convergence rate is typical for nonlinear dispersive equations and matches the best-known results
for DG methods applied to KdV-type equations with nonlinear flux; see, e.g., \cite{XuShu2007,zhang2020discontinuous}.
For HDG methods, existing analyses cover solvers for nonlinear KdV equations \cite{Samii2016}
and optimally convergent formulations for KdV-type models in linearized settings \cite{ChenCockburnDong2016, ChenDongJiang2018,Dong2017}.
The present work develops a stability and error analysis for the full Ostrovsky equation \eqref{eq:ostrovsky}, in which the
nonlinear convective flux must be controlled simultaneously with a nonlocal antiderivative (Coriolis) term. This combination, together with the bounded-interval boundary constraint required for uniqueness of $\partial_x^{-1}$, has not been addressed in prior HDG analyses to our knowledge.

In summary, the main contributions of this paper are:
\begin{itemize}
  \item We propose the first HDG spatial discretization for the full Ostrovsky equation \eqref{eq:ostrovsky}, simultaneously treating the third-order dispersive term and the nonlocal antiderivative term. For time integration we employ an implicit $\theta$-scheme where $\theta \in [\frac{1}{2},1]$.
  \item We derive a mixed first-order formulation that localizes the nonlocal operator via an auxiliary variable and incorporates boundary constraints needed to make $\partial_x^{-1}$ well-defined on bounded domains. The resulting HDG method is efficiently implementable by static condensation: element unknowns are eliminated locally and the globally coupled system involves only single-valued trace unknowns on the mesh skeleton.
  \item We establish an energy stability estimate for the semi-discrete method (and its fully discrete $\theta$-scheme counterpart) under assumptions on the stabilization parameters. Additionally, we provide the details on the error estimates of the fully discrete scheme.
  \item We provide an \(L^2\)-error analysis that includes the nonlinear flux contribution, yielding the convergence bound \(\|u-u_h\|_{L^2(\Omega)} = O(h^{k+1/2})\) for polynomial degree \(k\ge1\).
  \item Numerical tests confirm the optimal accuracy for smooth solutions and demonstrate robustness in challenging regimes, including solitary-wave propagation and non-smooth corner-type profiles, as well as the singular limit \(\beta\to0\) toward reduced OH dynamics.
\end{itemize}

Throughout, \(C>0\) denotes a generic constant independent of the mesh size \(h\) (and of \(\Delta t\) in fully discrete estimates), whose value may change from line to line.

The remainder of the paper is organized as follows.
In Section \ref{sec:prelim} we introduce notation, meshes, and the HDG finite element spaces.
Section \ref{sec:hdg} derives the mixed first-order formulation and presents the HDG spatial discretization, together with the energy stability estimate under appropriate stabilization conditions.
In Section \ref{sec:error}, we prove the semi-discrete error estimate.
Section \ref{sec:time} analyzes the fully discrete scheme based on the implicit $\theta$-method and establishes the corresponding stability and error bounds.
Section \ref{sec:numerics} reports numerical experiments that verify the theoretical convergence rates and demonstrate robustness for solitary-wave propagation and non-smooth peaked profiles, including the singular limit regime $\beta\to 0$. Finally, Section \ref{sec:conclusion} concludes the paper and outlines possible extensions.

\section{Preliminaries and notation}\label{sec:prelim}
We consider the Ostrovsky equation \eqref{eq:ostrovsky} on a bounded interval $\Omega = (x_L, x_R) \subset \mathbb{R}$. 
Let $\mathcal{T}_h = \{I_i\}_{i=1}^{N}$ be a partition of $\Omega$ into $N$ subintervals 
$I_i = (x_{i-1}, x_i)$ with $x_0 = x_L$ and $x_N = x_R$. 
The mesh size is denoted by $h = \max_i |I_i|$, where $|I_i| = x_i - x_{i-1}$. We set $\mathcal{F}_h = \{x_0, x_1, \dots, x_N\}$ and $\mathcal{F}_h^{0} = \{x_1,\dots,x_{N-1}\}$, 
and boundary points $\mathcal{F}_h^{\partial} = \{x_L, x_R\}$.
For each element $I_i$, we denote its boundary by $\partial I_i = \{x_{i-1}, x_i\}$.
The outward unit normal $n$ on $\partial I_i$ is defined as $n = -1$ at $x_{i-1}$ and $n = 1$ at $x_i$.

Let $P_k(I)$ denote the space of polynomials of degree at most $k$ on an interval $I$. 
The discontinuous polynomial space is
\[
V_h^k = \{ v \in L^2(\Omega) : v|_{I_i} \in P_k(I_i), \; \forall I_i \in \mathcal{T}_h \}.
\]
The trace spaces are 
\[
M_h(g) = \{ \widehat{v} \in L^2(\mathcal{F}_h): \widehat{v}|_{\mathcal{F}_h^{\partial}} = g \}, \qquad \widetilde{M}_h^L:= L^2(\mathcal{F}_h\setminus \{x_R\}), \qquad \widetilde{M}_h^R:= L^2(\mathcal{F}_h\setminus \{x_L\}),
\]
For $u, v \in L^2(\Omega)$, define
\[
(u,v)_{I_i} =  \int_{I_i} u v \, dx, \quad (u,v)_{\mathcal{T}_h} = \sum_{i=1}^{N} (u,v)_{I_i}, \quad 
\|u\|_{\mathcal{T}_h} = (u,u)_{\mathcal{T}_h}^{1/2},
\]
and we set 
\[
\langle \varphi, \psi n \rangle_{\partial\mathcal{T}_h} = \sum_{i=1}^{N} \langle \varphi, \psi n \rangle_{\partial I_i} = \sum_{i=1}^{N}[ \varphi|_{I_i}(x_i) \psi|_{I_i}(x_i) - \varphi|_{I_i}(x_{i-1}) \psi|_{I_i}(x_{i-1})],
\]
for sufficiently smooth $\varphi, \psi $.
The corresponding norm is $\|\varphi\|_{\partial\mathcal{T}_h} = \langle \varphi, \varphi \rangle_{\partial\mathcal{T}_h}^{1/2}$.
 We can also define $\langle \varphi, \psi n \rangle_{\partial\mathcal{T}_h} = \langle \varphi, \psi n \rangle_{\mathcal I^-} +\langle \varphi, \psi n \rangle_{\mathcal I^+}$, where $\langle \varphi, \psi n \rangle_{\mathcal I^-} = \sum_{i=1}^{N}\varphi|_{I_i}(x_i) \psi|_{I_i}(x_i)$ and $\langle \varphi, \psi n \rangle_{\mathcal I^+} = -\sum_{i=1}^{N} \varphi|_{I_i}(x_{i-1}) \psi|_{I_i}(x_{i-1}) $.

Denote by $P : L^2(\Omega) \to P_k(\Omega)$ the standard $L^2$-projection. For any function $w\in L^2(\Omega)$, we define $P w\in V_h^k$ as the unique element satisfying
\[
(P w - w,\phi)_{I_i}=0 \qquad \forall \phi\in P_k(I_i),\; \forall I_i\in\mathcal{T}_h.
\]

\section{HDG formulation for the Ostrovsky equation}\label{sec:hdg}
On $\Omega=(x_L,x_R)$ the operator $\partial_x^{-1}$ is not uniquely defined without fixing a boundary constraint.
Throughout the paper we localize the nonlocal term by introducing an auxiliary variable $v$ such that
\[
v_x = u \quad \text{in }\Omega,
\]
and we impose the boundary constraint
\(
v(x_R,t)=v_R=0.
\)
This choice yields a unique localization of $\partial_x^{-1}u$ on $\Omega$ and allows the Ostrovsky equation \eqref{eq:ostrovsky} to be rewritten as a local mixed system amenable to the HDG framework.

\subsection{Mixed formulation and local boundary conditions}
 By introducing auxiliary variables $q:=u_x$ and  $p=\beta q_x$, we rewrite the Ostrovsky equation \eqref{eq:ostrovsky} as a first-order system 
\begin{align}\label{eq:first-order}
q = u_x,  \quad
p = \beta q_x, \quad
v_x = u, \quad 
u_t - p_x + f(u)_x - \gamma v = 0,
\end{align}
along with the initial condition is \(u(x,0)=u_0(x)\) for \(x\in\Omega\). We first consider the case $\beta >0$ and we set the following boundary conditions 
\begin{equation}
\begin{aligned}\label{eq:physical-bc}
u(x_L,t) &= u_L(t), \quad u(x_R,t) = u_R(t), \quad  v(x_R,t) = v_R(t), \quad
 q(x_L,t) = q_L(t),
\end{aligned}
\end{equation}
to the system \eqref{eq:first-order}.
 On each element \(I_i\), the system \eqref{eq:first-order} is a local boundary-value problem for the unknowns \((u,q,p,v)\). To obtain a well-posed local problem, four boundary conditions are required. We impose  
 \[u(x_{i-1}) = \widehat{u}_{i-1},\quad u(x_i) = \widehat{u}_i, \quad v(x_i) = \widehat{v}_i, \quad q(x_{i-1}) = \widehat{q}_{i-1}.
 \]
 This choice is consistent with the sign structure used in the energy estimate and yields well-posed local solvers if $\beta>0$ and $\gamma>0$.
Given sufficiently regular initial condition and boundary data, the local problem admits a unique solution. Let \((U,Q,P,V)\) be defined piecewise on each \(I_i\) by local solution.
Then \((U,Q,P,V)\) is a global solution of \eqref{eq:first-order}-\eqref{eq:physical-bc} in \(\Omega\) if and only if it satisfy the continuity of $U$, $Q$, and $P - f(U)$ at the interior nodes \(x_i\in\mathcal{F}_h^{0}\) together with the global boundary conditions \eqref{eq:physical-bc}. Thus, for these
characterizations, the boundary data provided above for the local problem on $I_i$ are the global unknowns and determined from the continuity conditions and global boundary conditions \eqref{eq:physical-bc}, and the system of equations for the global unknowns is square. Our HDG method mimics the above continuous structure.

\subsection{HDG scheme}
The local HDG formulation on \(I_i\) is: for given numerical traces on $\partial I_i$ 
\begin{align*}
    \widehat{u}_h(x_{i-1}) =: \widehat{u}_{h,i-1},\quad \widehat{u}_h(x_i) =: \widehat{u}_{h,i},\quad \widehat{v}_h(x_{i}) =: \widehat{v}_{h,i}, \quad  \widehat{q}_h(x_{i-1}) =: \widehat{q}_{h,i-1} ,
\end{align*}
find \((u_h, v_h,  p_h, q_h) \in [P_k(I_i)]^4\) such that for all test functions \((\phi_u, \phi_v, \phi_p, \phi_q) \in [P_k(I_i)]^4\), there holds
\begin{subequations}\label{eq:hdg-local}
\begin{align}
((u_h)_t, \phi_u)_{I_i} + (p_h, \partial_x\phi_u)_{I_i} - (f(u_h), \partial_x\phi_u)_{I_i} - \gamma(v_h, \phi_u)_{I_i} \nonumber \\
- \langle( \widehat{p}_h - \widehat{f(u_h)})\,n, \phi_u \rangle_{\partial I_i} &= 0,\label{eq:hdg-local1}\\
-(v_h, \partial_x\phi_v)_{I_i} - (u_h, \phi_v)_{I_i} + \langle \widehat{v}_h n, \phi_v \rangle_{\partial I_i} &= 0, \label{eq:hdg-local2} \\
(p_h, \phi_p)_{I_i} + \beta (q_h, \partial_x\phi_p)_{I_i} - \beta \langle \widehat{q}_h n, \phi_p \rangle_{\partial I_i} &= 0, \label{eq:hdg-local3} \\
(q_h, \phi_q)_{I_i} + (u_h, \partial_x\phi_q)_{I_i} - \langle \widehat{u}_h n, \phi_q \rangle_{\partial I_i} &= 0. \label{eq:hdg-local4}
\end{align}
\end{subequations}
The remaining numerical traces are defined as follows:
    \begin{equation}\label{qv_nume_trace}
            \begin{aligned}
             \widehat{v}_h = v_h + \tau_{vq}(\widehat q_h-q_h)n, &\qquad \text{ at }~ x_{i-1} ,\\
       \widehat{q}_h = q_h +\tau_{qv}(\widehat v_h-v_h)n, &\qquad \text{ at }~ x_{i},
        \end{aligned}
    \end{equation}
        \begin{equation}\label{num_f_trace}
        \begin{aligned}
    \widehat{p}_h &=p_h+ \tau_{pu}\bigl(\widehat{u}_h - u_h\bigr)n, &\qquad\text{on } \partial I_i,\\
                \widehat{f(u_h)} &= f(u_h) - \tau_f(\widehat u_h-u_h)n,
&\qquad\text{on } \partial I_i.
        \end{aligned}
        \end{equation}        
The global HDG scheme is obtained by assembling the local formulations and imposing the transmission conditions weakly:
\begin{subequations}\label{eq:global-transmission}
\begin{align}
&\bigl\langle\widehat{v}_h,\mu_v\,n\bigr\rangle_{\partial\mathcal T_h}=\bigl\langle v_R,\mu_v\,n\bigr\rangle_{\{x_R\}},\label{eq:trans-v}\\
&\displaystyle\bigl\langle\widehat{q}_h,\mu_q\,n\bigr\rangle_{\partial\mathcal T_h}=\bigl\langle q_L,\mu_q\,n\bigr\rangle_{\{x_L\}},
\label{eq:trans-q}\\
&\bigl\langle\widehat{p}_h-\widehat{f(u_h)},\mu_p\,n\bigr\rangle_{\partial\mathcal T_h}=0,\label{eq:trans-pf}
\end{align}
\end{subequations}
for all $\mu_v \in \widetilde M_h^L$, $\mu_q \in \widetilde M_h^R$, and $\mu_v \in  M_h(0)$. The initial condition is enforced by $u_h(\cdot,0)=Pu_0$ and the global boundary conditions
\begin{equation}
\begin{aligned}\label{eq:physical-bc_num}
u_h(x_L,t) &= u_L(t), \quad u_h(x_R,t) = u_R(t), \quad  v_h(x_R,t) = v_R(t), \quad
 q_h(x_L,t) = q_L(t).
\end{aligned}
\end{equation}
The globally unknown numerical traces $\widehat{u}_h \in M_h(0)$, $\widehat{v}_h \in \widetilde M_h^R$, and $\widehat{q}_h \in \widetilde M_h^L$ can be determined by solving the above weak transmission conditions \eqref{eq:global-transmission}. Note that, $\widehat{u}$ is single–valued on $\mathcal{F}$, thus $\langle \widehat u_h n, \delta \rangle_{\partial\mathcal T_h} = 0$ for all $\delta\in M_h(0)$.

The system \eqref{eq:hdg-local}--\eqref{eq:global-transmission} is square. Under appropriate choices of the stabilization parameters $\tau_{vq},\tau_{qv}, \tau_{pu},\tau_f$, we expect unique solvability.

\subsection{Stability analysis}\label{sec:stability}
We now establish the stability of the semi‑discrete HDG scheme under the homogeneous boundary conditions \(
u_L(t)=u_R(t)= v_R(t) =q_L(t)=0.\)
Define the discrete energy
\[
\mathcal{E}_h(t)=\frac12\,\|u_h(t)\|_{\mathcal{T}_h}^2 .
\]
To control the nonlinear term, we introduce the following stabilization function. For any oriented point \(x\in\partial\mathcal{T}_h\) we set
\begin{equation}\label{eq:tildetau}
\tilde\tau(\widehat{u}_h,u_h)
:=\frac{1}{(\widehat{u}_h-u_h)^{2}}
   \int_{\widehat{u}_h}^{u_h}\bigl(f(s)-f(u_h)\bigr)n\,\mathrm ds.
\end{equation}
A straightforward estimate gives
\begin{equation}\label{tildetau}
\bigl|\tilde\tau(\widehat{u}_h,u_h)(x)\bigr|
\;\le\;
\frac12\sup_{s\in J(\widehat{u}_h(x),u_h(x))}|f'(s)|,
\end{equation}
with \(J(a,b)=[\min(a,b),\max(a,b)]\). In order to ensure stability, we need the point-wise inequality
\(
  \tau_f-\tilde\tau\ge0
\)
(the first component of our collective stabilization function), for which it
is sufficient to choose, e.g.
\(
  \tau_f=\sup_{s\in J(\widehat u_h,u_h)}\tfrac12\bigl\lvert f'( s)\bigr\rvert+\varepsilon
\) for some $\varepsilon>0$. In addition, we require the following global assumption.
\begin{assumption}[Stabilization Parameters]\label{ass:stab}
The stabilization parameters \(\tau_f, \tau_{pu}, \tau_{vq}, \tau_{qv}\) are chosen such that the following inequalities hold pointwise on \(\partial\mathcal{T}_h\):
\[
\tau_f - \tilde{\tau} \geq \tilde{C}, \quad 
\tau_{pu} \geq c, \quad 
\frac{\beta}{2} - \frac{\gamma}{2}\tau_{vq}^2 \geq c, \quad 
\frac{\gamma}{2} - \frac{\beta}{2}\tau_{qv}^2 \geq c.
\]
Here, \(\tilde{C}\) is chosen sufficiently large, and \(c > 0\) is chosen such that \(c - \bar{C}\delta \geq 0\), where \(\bar{C}\) is a constant arising in the proof of Theorem \ref{thm:error} and \(\delta > 0\) is a small parameter from Young's inequality. For the {\it stability} statement alone it suffices to take $\widetilde C=c=0$.
\end{assumption}

\begin{theorem}[Stability]\label{thm:stability}
Under the homogeneous boundary conditions stated above and Assumption~\ref{ass:stab}, the solution of the HDG scheme \eqref{eq:hdg-local} satisfies
\[
\frac{d}{dt}\mathcal{E}_h(t)\le0 .
\]
\end{theorem}
\begin{proof}
We begin by choosing the test functions
\[
\phi_u=u_h,\quad \phi_v=-\gamma v_h,\quad \phi_p=-q_h,\quad \phi_q=p_h,
\]
in the global equations \eqref{eq:hdg-local1}--\eqref{eq:hdg-local4}, adding these four equations and summing on all $I_i$ yields
\begin{align*}
&\frac12\frac{d}{dt}\|u_h\|_{\mathcal{T}_h}^2
+(p_h,\partial_x u_h)_{\mathcal{T}_h}+(u_h,\partial_x p_h)_{\mathcal{T}_h}
-\bigl\langle\widehat{p}_h,u_h n\bigr\rangle_{\partial\mathcal{T}_h}
-\bigl\langle\widehat{u}_h,p_h n\bigr\rangle_{\partial\mathcal{T}_h} \\
&-(f(u_h),\partial_x u_h)_{\mathcal{T}_h}
+\bigl\langle\widehat{f(u_h)},u_h n\bigr\rangle_{\partial\mathcal{T}_h}
+\gamma(v_h,\partial_x v_h)_{\mathcal{T}_h}
-\gamma\bigl\langle\widehat{v}_h,v_h n\bigr\rangle_{\partial\mathcal{T}_h} \\
&-\beta(q_h,\partial_x q_h)_{\mathcal{T}_h}
+\beta\bigl\langle\widehat{q}_h,q_h n\bigr\rangle_{\partial\mathcal{T}_h}=0 .
\end{align*}
Using integration by parts and the transmission conditions \eqref{eq:global-transmission} we obtain the identity
\begin{align*}
&\frac12\frac{d}{dt}\|u_h\|_{\mathcal{T}_h}^2
+\bigl\langle(\widehat{p}_h-p_h),(\widehat{u}_h-u_h)n\bigr\rangle_{\partial\mathcal{T}_h}
-(f(u_h),\partial_x u_h)_{\mathcal{T}_h}
-\bigl\langle\widehat{f(u_h)},(\widehat{u}_h-u_h)n\bigr\rangle_{\partial\mathcal{T}_h} \\
&\qquad
-\frac\gamma2\bigl(-\widehat{v}_h(x_L)^2+\widehat{v}_h(x_R)^2\bigr)
+\frac\gamma2\bigl\langle(\widehat{v}_h-v_h)^2,n\bigr\rangle_{\partial\mathcal{T}_h} \\
&\qquad
+\frac\beta2\bigl(-\widehat{q}_h(x_L)^2+\widehat{q}_h(x_R)^2\bigr)
-\frac\beta2\bigl\langle(\widehat{q}_h-q_h)^2,n\bigr\rangle_{\partial\mathcal{T}_h}=0 .
\end{align*}
Let \(G\) be an antiderivative of \(f\). Then
\[
(f(u_h),\partial_x u_h)_{\mathcal{T}_h}
= \bigl\langle G(u_h),n\bigr\rangle_{\partial\mathcal{T}_h}
= \Bigl\langle\int_{\widehat{u}_h}^{u_h}f(s)\,\mathrm{d}s,n\Bigr\rangle_{\partial\mathcal{T}_h}.
\]
Combining this with the definition of \(\widehat{f(u_h)}\) from \eqref{num_f_trace}, we have
\begin{align*}
&-(f(u_h),\partial_x u_h)_{\mathcal{T}_h}
-\bigl\langle\widehat{f(u_h)},(\widehat{u}_h-u_h)n\bigr\rangle_{\partial\mathcal{T}_h} \\
&\qquad=
-\Bigl\langle\int_{\widehat{u}_h}^{u_h}\bigl(f(s)-f(u_h)\bigr)\mathrm{d}s,n\Bigr\rangle_{\partial\mathcal{T}_h}
-\bigl\langle (\widehat{f(u_h)}- f(u_h) ) ,(\widehat{u}_h-u_h)n\bigr\rangle_{\partial\mathcal{T}_h} \\
&\qquad=\bigl\langle(\tau_f-\tilde\tau)(\widehat{u}_h-u_h)^2,1\bigr\rangle_{\partial\mathcal{T}_h}.
\end{align*}
where $\tilde \tau$ is defined by \eqref{tildetau}.
We insert the definitions of the numerical traces \(\widehat{p}_h,\widehat{v}_h,\widehat{q}_h\) from \eqref{qv_nume_trace}-\eqref{num_f_trace} and apply the boundary conditions, which gives
\begin{equation}
\begin{aligned}\label{identity_31}
&\frac12\frac{d}{dt}\|u_h\|_{\mathcal{T}_h}^2+\tau_{pu}\bigl\langle(\widehat u_h-u_h)^2,1\bigr\rangle_{\partial \mathcal T_h}  +\langle \tau_f-\tilde \tau, (\widehat u_h -u_h)^2\rangle_{\partial \mathcal T_h}\\& \quad
+\frac\gamma2\widehat{v}_h(x_L)^2 +\frac\gamma2\tau_{vq}^2\langle( \widehat{q}_h- q_h)^2, n\rangle_{\mathcal I^+} +\frac\gamma2\langle( \widehat{v}_h- v_h)^2, n\rangle_{\mathcal I^-} \\&\quad+\frac\beta2\widehat{q}_h(x_R)^2 -\frac\beta2\langle( \widehat{q}_h- q_h)^2, n\rangle_{\mathcal I^+}-\frac\beta2\tau_{qv}^2\langle( \widehat{v}_h- v_h)^2, n\rangle_{\mathcal I^-}=0.
\end{aligned}
\end{equation}
Applying assumption \ref{ass:stab} implies
\begin{equation*}
   \frac12\frac{d}{dt}\|u_h\|_{\mathcal{T}_h}^2\leq 0.
\end{equation*}
This completes the proof.
\end{proof}
\begin{remark}[Conservative property for periodic problems with $\beta>0$ and $\gamma>0$]
For periodic boundary conditions ($\widehat{v}_h(x_L) =\widehat{v}_h(x_R)$, $\widehat{q}_h(x_L) = \widehat{q}_h(x_R)$, and $\widehat{u}_h(x_L) = \widehat{u}_h(x_R)$), our HDG scheme \eqref{eq:hdg-local}-\eqref{eq:global-transmission} can be made exactly conservative by an appropriate choice of the stabilization parameters. As we mentioned in Assumption \ref{ass:stab} that we can take $\tilde C =c =0$ for the stability analysis. Therefore, by setting $\tau_{pu}=0$, $\tau_{vq}=\sqrt{\beta/\gamma}$, $\tau_{qv}=\sqrt{\gamma/\beta}$, and $\tau_f = \tilde{\tau}$ (pointwise), \eqref{identity_31} yields
\[
\frac12\frac{d}{dt}\|u_h\|_{\mathcal{T}_h}^2 = 0,
\]
making the scheme energy-conservative.

In practice we take $\tau_f=\tilde\tau$, where $\tilde\tau$ is evaluated elementwise from the current trace and 
if $\widehat u_h=u_h$ at a quadrature point, we interpret the quotient-based definition of $\tilde\tau$ in the limiting sense (equivalently, by replacing the quotient with $|f'(u_h)|$), which is well-defined for the smooth fluxes considered here.
\end{remark}
\begin{remark}[Case $\beta < 0$]
The stability analysis and numerical traces presented above assume $\beta > 0$. For the case $\beta<0$, the stability analysis requires a different choice of numerical traces and boundary conditions for the auxiliary variable $q$. Specifically, we now set local given traces on $\partial I_i$ as
\[
\widehat{u}_h(x_{i-1}) =: \widehat{u}_{h,i-1},\quad \widehat{u}_h(x_i) =: \widehat{u}_{h,i},\quad \widehat{v}_h(x_{i}) =: \widehat{v}_{h,i}, \quad  \widehat{q}_h(x_{i}) =: \widehat{q}_{h,i},
\]
and impose the global boundary condition $u(x_L,t) = u_L(t),~ u(x_R,t) = u_R(t),~  v(x_R,t) = v_R(t)$, and $
 q(x_R,t) = q_R(t)$ . The remaining local numerical traces are accordingly modified to 
 $$ \widehat{v}_h = v_h \qquad\text{and} \qquad \widehat{q}_h = q_h, \qquad \text{ at }~ x_{i-1},$$
along with other numerical traces \eqref{num_f_trace},
 while the transmission condition for $q$ becomes
\[
\bigl\langle\widehat{q}_h,\mu_q\,n\bigr\rangle_{\partial\mathcal T_h}
= \bigl\langle q_R,\mu_q\,n\bigr\rangle_{\{x_R\}}.
\]
Proceeding with the same energy argument as in the proof of Theorem~\ref{thm:stability}, the energy identity \eqref{identity_31} becomes
\begin{align*}
&\frac12\frac{d}{dt}\|u_h\|_{\mathcal{T}_h}^2
+ \tau_{pu}\bigl\langle(\widehat u_h-u_h)^2,1\bigr\rangle_{\partial\mathcal T_h}
+ \langle \tau_f-\tilde \tau, (\widehat u_h -u_h)^2\rangle_{\partial\mathcal T_h} \\
&\quad + \frac{\gamma}{2}\widehat{v}_h(x_L)^2 
+ \frac{\gamma}{2}\bigl\langle( \widehat{v}_h- v_h)^2, n\bigr\rangle_{\mathcal I^-} - \frac{\beta}{2}\widehat{q}_h(x_L)^2 
- \frac{\beta}{2}\bigl\langle( \widehat{q}_h- q_h)^2, n\bigr\rangle_{\mathcal I^-}=0.
\end{align*}
Under the same homogeneous boundary conditions, stability is ensured if the stabilization parameters satisfy
\(
\tau_f - \tilde{\tau} \geq \tilde{C},~\;
\tau_{pu} \geq c, 
\)
with $\tilde{C}\geq 0$ and $c\geq 0$ as in Assumption \ref{ass:stab}.
\end{remark}
\section{Error analysis}\label{sec:error}

We now present an error analysis for the proposed HDG method \eqref{eq:hdg-local}. The analysis is based on standard $L^2$ projections and follows the usual approach for error estimation of discontinuous Galerkin methods. Let $(u,v,p,q)$ denote the exact solution of the OV equation \eqref{eq:ostrovsky}  with sufficient regularity, and let $(u_h,v_h,p_h,q_h)$ be the semi-discrete HDG approximation defined by \eqref{eq:hdg-local}. 

We begin by considering the  $L^2$ projection onto the local polynomial space $P_k(I_i)$. 
\begin{lemma}[Inverse inequalities \cite{Ciarlet1978Fem,XuShu2007}]\label{lem:inverse_2d}
For every $v_h\in V_h^{k}$ $(k\ge0)$, there exists a positive constant $C$ independent
of $v_h$ and $h$, such that the following estimates hold
\begin{equation}\label{invesre}
\begin{aligned}
\|\partial_x v_h\|_{\mathcal T_h}
  \le Ch^{-1}\|v_h\|_{\mathcal T_h},\quad
\|v_h\|_{\partial \mathcal T_h}
  \le Ch^{-\frac{1}{2}}\|v_h\|_{\mathcal T_h}, \quad \|v_h\|_{L^\infty(\mathcal T_h)}
  \le Ch^{-\frac{1}{2}}\|v_h\|_{\mathcal T_h}.
  \end{aligned}
\end{equation}

\end{lemma}

\begin{lemma}[Interpolation inequality \cite{Ciarlet1978Fem,XuShu2007}]\label{lem:approx_2d}
For any $\omega\in H^{k+1}(\mathcal{T}_h)$ there exists a constant
$C>0$, 
independent of $h$, such that
\begin{equation}\label{eq:proj-est}
     \|P\omega-\omega\|_{\mathcal T_h}
  +
  h^{\frac12}\|P\omega-\omega\|_{L^\infty(\mathcal T_h)}
  +
  h^{\frac12}\|P\omega-\omega\|_{\partial\mathcal T_h}
  \le
  Ch^{k+1}.
\end{equation}
\end{lemma}

To deal with the nonlinear flux, we make an {\it a priori} assumption, that for all $t< T$ and small enough $h$, there holds 
    \begin{equation}\label{ass1uu_h}
        \begin{aligned}
            \|u-u_h\|_{\mathcal T_h} \leq h.
        \end{aligned}
    \end{equation}
    Consequently, we have
     \begin{equation}\label{ass1uu_h2}
        \begin{aligned}
            &\|u-u_h\|_{L^{\infty}(\mathcal T_h)} \leq Ch^{1/2}.
        \end{aligned}
    \end{equation}
The above assumption is not required for linear flux $f(u) =cu$. 

We decompose the errors into projection errors and the errors between the projections and the numerical solutions
\begin{align*}
e_\omega:= \omega-\omega_h &= (P\omega-\omega_h) - (P\omega-\omega) =: \xi^\omega - \rho^\omega,\qquad\qquad \omega \in \{u,v,p,q\}.
\end{align*}
On the skeleton $\partial \mathcal T_h$, we introduce
\begin{align}\label{define_hatxi}
\widehat{\xi}^u &:= u-\widehat{u}_h, \quad \widehat{\xi}^q := q-\widehat{q}_h, \quad \widehat{\xi}^p := p-\widehat{p}_h, \quad \widehat{\xi}^v := v-\widehat{v}_h,
\end{align}
so that the simple algebraic manipulations and definitions \eqref{qv_nume_trace}-\eqref{num_f_trace} imply that in $I_i$
\begin{equation}\label{qv_nume_tracexi}
            \begin{aligned}
     \widehat{\xi}^v &= \xi^v + \tau_{vq}(\widehat \xi^q -\xi^q)n-(\rho^v-\tau_{vq}\rho^qn), \quad \text{ at }~ x_{i-1},  \\ 
    \widehat{\xi}^q &= \xi^q + \tau_{qv}(\widehat \xi^v -\xi^v)n-(\rho^q-\tau_{qv}\rho^vn), \quad \text{ at }~ x_{i},
        \end{aligned}
    \end{equation}
    and
 \begin{equation}\label{num_p_tracexi}
            \begin{aligned}
            \widehat{\xi}^p =
                \xi^p+ \tau_{pu}\bigl(\widehat \xi^u -\xi^u\bigr)n  - (\rho^p- \tau_{pu}\rho^un ),\qquad~ \text{ at }~ x_{i-1} ~\text{ and }~ x_{i} ,    \end{aligned}
    \end{equation}    
Since the exact solution satisfies \eqref{eq:hdg-local}-\eqref{eq:global-transmission}, the HDG scheme \eqref{eq:hdg-local} yields the error equations
\begin{subequations}\label{eq:error-global}
\begin{align}
((e_u)_t,\phi_u)_{\mathcal{T}_h}+(e_p,\partial_x\phi_u)_{\mathcal{T}_h}-(f(u)-f(u_h),\partial_x\phi_u)_{\mathcal{T}_h}-\gamma(e_v,\phi_u)_{\mathcal{T}_h}\nonumber\\
\qquad\qquad -\langle (p-\widehat{p}_h)-(f(u)-\widehat{f(u_h)}),\phi_u\,n\rangle_{\partial\mathcal{T}_h}&=0,\label{eq:error1}\\
-(e_v,\partial_x\phi_v)_{\mathcal{T}_h}-(e_u,\phi_v)_{\mathcal{T}_h}+\langle v-\widehat{v}_h,\phi_v\,n\rangle_{\partial\mathcal{T}_h}&=0,\label{eq:error2}\\
(e_p,\phi_p)_{\mathcal{T}_h}+\beta(e_q,\partial_x\phi_p)_{\mathcal{T}_h}-\beta\langle q-\widehat{q}_h,\phi_p\,n\rangle_{\partial\mathcal{T}_h}&=0,\label{eq:error3}\\
(e_q,\phi_q)_{\mathcal{T}_h}+(e_u,\partial_x\phi_q)_{\mathcal{T}_h}-\langle u-\widehat{u}_h,\phi_q\,n\rangle_{\partial\mathcal{T}_h}&=0.\label{eq:error4}
\end{align}
\end{subequations}
Using the decomposition $e_\omega=\xi^\omega-\rho^\omega$ ($\omega\in\{u,q,p,v\}$), definition on $L^2$ projection $\rho^\omega$ and \eqref{define_hatxi}, we obtain the error equations in terms of $\xi$ and $\rho$:
\begin{subequations}\label{eq:error-proj}
\begin{align}
((\xi^u)_t,\phi_u)_{\mathcal{T}_h}+(\xi^p,\partial_x\phi_u)_{\mathcal{T}_h}
-(f(u)-f(u_h),\partial_x\phi_u)_{\mathcal{T}_h}-\gamma(\xi^v,\phi_u)_{\mathcal{T}_h}\nonumber\\
\qquad -\bigl\langle \widehat{\xi}^p , \phi_u n\bigr\rangle_{\partial\mathcal{T}_h}+\langle (f(u)-\widehat{f(u_h)}),\phi_u\,n\rangle_{\partial\mathcal{T}_h}
&=0\label{eq:error-proj1}\\
-(\xi^v,\partial_x\phi_v)_{\mathcal{T}_h}-(\xi^u,\phi_v)_{\mathcal{T}_h}+\langle \widehat{\xi}^v,\phi_v n\rangle_{\partial\mathcal{T}_h}
&=0,\label{eq:error-proj2}\\[4pt](\xi^p,\phi_p)_{\mathcal{T}_h}+\beta(\xi^q,\partial_x\phi_p)_{\mathcal{T}_h}-\beta\langle \widehat{\xi}^q,\phi_p n\rangle_{\partial\mathcal{T}_h}
&=0,\label{eq:error-proj3}\\[4pt]
(\xi^q,\phi_q)_{\mathcal{T}_h}+(\xi^u,\partial_x\phi_q)_{\mathcal{T}_h}-\langle \widehat{\xi}^u,\phi_q n\rangle_{\partial\mathcal{T}_h}
&=0,\label{eq:error-proj4}
\end{align}
\end{subequations}
and transmission condition \eqref{eq:global-transmission} and smoothness of exact solution $(u,v,p,q)$ gives
\begin{subequations}\label{eq:global-transmissionxi}
\begin{align}
\displaystyle\bigl\langle\widehat{\xi}^v,\mu_v\,n\bigr\rangle_{\partial\mathcal T_h}=0,\qquad \displaystyle\bigl\langle\widehat{\xi}^q,\mu_q\,n\bigr\rangle_{\partial\mathcal T_h}=0,\qquad\bigl\langle\widehat{\xi^p}-(f(u) -\widehat{f(u_h)}),\mu_p\,n\bigr\rangle_{\partial\mathcal T_h}=0,
\end{align}
\end{subequations}
for all $\mu_v \in \widetilde M_h^L$, $\mu_q \in \widetilde M_h^R$, and $\mu_v \in  M_h(0)$.
These equations form the starting point for the subsequent error estimates. We now state and prove the main error estimate. We assume that the exact solution $(u,v,p,q)$ is sufficiently smooth, specifically that each component belongs to $H^{k+1}(\Omega)$ for the spatial regularity and is continuously differentiable in time. 

\begin{theorem}\label{thm:error}
Under the above smoothness assumptions on exact tuple solution $(u,v,p,q)$ and with the stabilization parameters chosen according to Assumption \ref{ass:stab}, the error between the exact solution and the semi-discrete HDG solution $(u_h,v_h,p_h,q_h)$ satisfies
\begin{equation}\label{errest}
\|u-u_h\|_{L^2(\Omega)} \le C h^{k+1/2}, \qquad k\geq 1,
\end{equation}
provided $h$ is sufficiently small, where constant $C$ may depend on the exact solution, the polynomial degree $k$, the parameters $\beta,\gamma$, the final time $T$, and the mesh regularity, but is independent of the mesh size $h$.
\end{theorem}
\begin{proof}
We follow the energy argument used in the stability analysis. Choose the test functions in \eqref{eq:error-proj} as
\[
\phi_u = \xi^u,\quad \phi_v = -\gamma \xi^v,\quad \phi_p = -\xi^q,\quad \phi_q = \xi^p.
\]
Adding the four equations \eqref{eq:error-proj1}--\eqref{eq:error-proj4} with these test functions and integration by parts yields
\begin{align*}
&\frac12\frac{d}{dt}\|\xi^u\|_{\mathcal{T}_h}^2 +\bigl\langle \xi^p,\xi^u n \bigr\rangle_{\partial\mathcal{T}_h}  
- \bigl\langle \widehat{\xi}^p,\xi^u n \bigr\rangle_{\partial\mathcal{T}_h}- \bigl\langle \widehat{\xi}^u,\xi^p n \bigr\rangle_{\partial\mathcal{T}_h}  +\frac\gamma2\bigl\langle \xi^v,\xi^v n \bigr\rangle_{\partial\mathcal{T}_h} - \gamma\bigl\langle \widehat{\xi}^v,\xi^v n \bigr\rangle_{\partial\mathcal{T}_h} \\
&\quad- \frac\beta2\bigl\langle \xi^q,\xi^q n \bigr\rangle_{\partial\mathcal{T}_h}  + \beta\bigl\langle \widehat{\xi}^q,\xi^q n \bigr\rangle_{\partial\mathcal{T}_h}   = \bigl(f(u)-f(u_h),\partial_x\xi^u\bigr)_{\mathcal{T}_h}-\bigl\langle f(u)-\widehat{f(u_h)},\xi^u n \bigr\rangle_{\partial\mathcal{T}_h}.
\end{align*}
Now using the conditions \eqref{eq:global-transmissionxi}, we obtain
\begin{align*}
&\frac12\frac{d}{dt}\|\xi^u\|_{\mathcal{T}_h}^2 
+ \bigl\langle (\widehat{\xi}^p-\xi^p),(\widehat \xi^u-\xi^u) n \bigr\rangle_{\partial\mathcal{T}_h} +\frac\gamma2 \widehat \xi^v(x_L)^2+\frac\gamma2\bigl\langle (\widehat\xi^v-\xi^v)^2, n \bigr\rangle_{\partial\mathcal{T}_h} \\
&\quad+\frac\beta2 \widehat \xi^q(x_R)^2-\frac\beta2\bigl\langle (\widehat\xi^q-\xi^q)^2, n \bigr\rangle_{\partial\mathcal{T}_h}  = \bigl(f(u)-f(u_h),\partial_x\xi^u\bigr)_{\mathcal{T}_h}+\bigl\langle f(u)-\widehat{f(u_h)},(\widehat \xi^u -\xi^u)n \bigr\rangle_{\partial\mathcal{T}_h}.
\end{align*}
Utilizing definitions of traces \eqref{qv_nume_tracexi}-\eqref{num_p_tracexi} and \eqref{num_f_trace}, we get
\begin{align*}
&\frac12\frac{d}{dt}\|\xi^u\|_{\mathcal{T}_h}^2 +\tau_{pu}\bigl\langle(\widehat \xi^u-\xi^u)^2,1\bigr\rangle_{\partial\mathcal{T}_h}  
+\frac\gamma2\widehat{\xi}^v(x_L)^2 +\frac\gamma2\langle(\widehat\xi^v-\xi^v)^2, n\rangle_{\mathcal I^-} +\frac\gamma2\tau_{vq}^2\langle(\widehat\xi^q-\xi^q)^2, n\rangle_{\mathcal I^+}\\&\qquad +\frac\beta2\widehat{\xi}^q(x_R)^2 -\frac\beta2\langle(\widehat\xi^q-\xi^q)^2, n\rangle_{\mathcal I^+} -\frac\beta2\tau_{qv}^2\langle(\widehat\xi^v-\xi^v)^2, n\rangle_{\mathcal I^-}\\&=\bigl(f(u)-f(u_h),\partial_x\xi^u\bigr)_{\mathcal{T}_h}+\bigl\langle f(u)-f(u_h),(\widehat \xi^u -\xi^u)n \bigr\rangle_{\partial\mathcal{T}_h} + \bigl\langle \tau_f (\widehat u_h-u_h),(\widehat \xi^u -\xi^u)\bigr\rangle_{\partial\mathcal{T}_h}\\& \qquad + \bigl\langle (\rho^p- \tau_{pu}\rho^un), (\widehat \xi^u -\xi^u)n\bigl\rangle_{\partial\mathcal{T}_h}  +\frac\beta2\langle(\rho^q-\tau_{qv}\rho^v)^2, n\rangle_{\mathcal I^-} - \frac\gamma2\langle(\rho^v-\tau_{vq}\rho^q)^2, n\rangle_{\mathcal I^+} \\& \qquad- \beta \langle(\rho^q-\tau_{qv}\rho^v), \tau_{qv}(\widehat\xi^v-\xi^v)n\rangle_{\mathcal I^-}   + \gamma \langle(\rho^v-\tau_{vq}\rho^q), \tau_{vq}(\widehat\xi^q-\xi^q)n\rangle_{\mathcal I^+}.
\end{align*}
Employing the Assumption \ref{ass:stab} and ignoring some positive terms from the left hand side, we have 
\begin{align}\label{auxineq1}
&\nonumber\frac12\frac{d}{dt}\|\xi^u\|_{\mathcal{T}_h}^2 +c(\|\widehat \xi^u -\xi^u\|_{\partial\mathcal T_h}^2 +  \|\widehat \xi^q -\xi^q\|_{\mathcal I^-}^2 + \|\widehat \xi^v -\xi^v\|_{\mathcal I^+}^2) \leq \bigl(f(u)-f(u_h),\partial_x\xi^u\bigr)_{\mathcal{T}_h} \\& \nonumber \qquad+\bigl\langle f(u)-f(u_h),(\widehat \xi^u -\xi^u)n \bigr\rangle_{\partial\mathcal{T}_h} - \bigl\langle \tau_f (\widehat \xi^u-\xi^u +\rho^u),(\widehat \xi^u -\xi^u)\bigr\rangle_{\partial\mathcal{T}_h} \\& \nonumber\qquad + \bigl\langle (\rho^p- \tau_{pu}\rho^un), (\widehat \xi^u -\xi^u)n\bigl\rangle_{\partial\mathcal{T}_h}  +\frac\beta2\langle(\rho^q-\tau_{qv}\rho^v)^2, n\rangle_{\mathcal I^-} - \frac\gamma2\langle(\rho^v-\tau_{vq}\rho^q)^2, n\rangle_{\mathcal I^+} \\& \qquad- \beta \langle(\rho^q-\tau_{qv}\rho^v), \tau_{qv}(\widehat\xi^v-\xi^v)n\rangle_{\mathcal I^-}   + \gamma \langle(\rho^v-\tau_{vq}\rho^q), \tau_{vq}(\widehat\xi^q-\xi^q)n\rangle_{\mathcal I^+}.
\end{align}
Now we estimate the right hand side using the Taylor's expansion, projection estimates, and trace definition \eqref{num_f_trace}. We write
\begin{align*}
  f(u)-f(u_h)= f'(u)(u-u_h)-\frac12\,f''_{\!u}\,(u-u_h)^2
   =  f'(u)(\xi^u - \rho^u)-\frac12\,f''_{\!u}\,(\xi^u - \rho^u)^2, 
\end{align*}
so that the nonlinear part of the right hand side of the inequality \eqref{auxineq1} becomes
\begin{equation}\label{H_ij2}
    \begin{split}
       &\bigl(f(u)-f(u_h),\partial_x\xi^u\bigr)_{\mathcal{T}_h}+\bigl\langle f(u)-f(u_h),(\widehat \xi^u -\xi^u)n \bigr\rangle_{\partial\mathcal{T}_h} - \bigl\langle \tau_f (\widehat \xi^u-\xi^u +\rho^u),(\widehat \xi^u -\xi^u)\bigr\rangle_{\partial\mathcal{T}_h}
         \\& \quad= \underbrace{\bigl(f'(u)(\xi^u - \rho^u),\partial_x\xi^u\bigr)_{\mathcal T_h}  + \langle f'(u) (\xi^u-\rho^u),n(\widehat \xi^u -\xi^u)\rangle_{\partial \mathcal T_h}}_{f_1}\\&\qquad \underbrace{-\Bigl(\frac12\,f''_{\!u}\,(\xi^u - \rho^u)^2,\partial_x\xi^u\Bigr)_{\mathcal T_h}-\Big\langle \frac12\, f''_{\!u}\,
  ( \xi^u-\rho^u)^2,n(\widehat \xi^u -\xi^u)\Big\rangle_{\partial \mathcal T_h}}_{f_2}\\&\qquad \underbrace{-\langle \tau_f (\widehat \xi^u - \xi^u +\rho^u),(\widehat \xi^u-\xi^u)\rangle_{\partial \mathcal T_h}}_{f_3}.
    \end{split}
\end{equation}
Incorporating interpolation \eqref{eq:proj-est} and inverse inequalities \eqref{invesre}, integration by parts, Young's inequality, and an {\it a priori} assumption \eqref{ass1uu_h}  with  $|u-\bar u|_{I_i} = \mathcal O(h)$ on each $I_i$ for some constant $\bar u$, we have 
\begin{equation*}
    \begin{aligned}
        f_1 &= \bigl(f'(u)(\xi^u - \rho^u),\partial_x\xi^u\bigr)_{\mathcal T_h}  + \langle f'(u) (\xi^u-\rho^u),n(\widehat \xi^u -\xi^u)\rangle_{\partial \mathcal T_h} \\
        &= \Bigl(f'(u),\frac{1}{2}\partial_x(\xi^u)^2\Bigr)_{\mathcal T_h} -\bigl(f'(u) \rho^u,\partial_x\xi^u\bigr)_{\mathcal T_h} + \langle f'(u) (\xi^u-\rho^u),n(\widehat \xi^u -\xi^u)\rangle_{\partial \mathcal T_h}\\& = 
        -\Bigl( f''(u)u_x,\frac{1}{2}(\xi^u)^2\Bigr)_{\mathcal T_h} + \Bigl\langle (f'(u)-f'(\bar u)),n\frac{1}{2}(\xi^u)^2\Bigr\rangle_{\partial \mathcal T_h} +\Bigl\langle f'(\bar u),n\frac{1}{2}(\xi^u)^2\Bigr\rangle_{\partial \mathcal T_h} \\& \quad -\bigl( (f'(u)-f'(\bar u))\rho^u,\partial_x\xi^u\bigr)_{\mathcal T_h}  + \langle (f'(u)-f'(\bar u))( \xi^u-\rho^u),n(\widehat \xi^u -\xi^u)\rangle_{\partial \mathcal T_h} \\&\quad+ \langle f'(\bar u)( \xi^u-\rho^u),n(\widehat \xi^u -\xi^u)\rangle_{\partial \mathcal T_h}
        - \Bigl\langle f'(\bar u),n\frac{1}{2}(\widehat \xi^u)^2\Bigr\rangle_{\partial \mathcal T_h}
        \\& \leq C\|f''(u)\|_{L^{\infty}(\mathcal{T}_h)}\|u_x\|_{L^\infty(\mathcal T_h)}\|\xi^u\|_{\mathcal{T}_h}^2+ Ch^{-\frac{1}{2}}|f'(u)-f'(\bar u)|_{\partial \mathcal T_h}\|\xi^u\|_{\mathcal{T}_h}^2 \\& \quad 
        -\Bigl\langle f'(\bar u),n\frac{1}{2}(\xi^u)^2\Bigr\rangle_{\partial \mathcal T_h}
        - \Bigl\langle f'(\bar u),n\frac{1}{2}(\widehat \xi^u)^2\Bigr\rangle_{\partial \mathcal T_h} + \langle f'(\bar u) \xi^u,n\widehat \xi^u \rangle_{\partial \mathcal T_h}\\& \quad + h^{-1}|f'(u)-f'(\bar u)|_{\mathcal T_h}\|\rho^u\|_{\mathcal T_h}\|\xi^u\|_{\mathcal T_h}\\&\quad + h^{-\frac{1}{2}}|f'(u)-f'(\bar u)|_{\partial \mathcal T_h}(\|\rho^u\|_{\mathcal T_h}+\|\xi^u\|_{\mathcal T_h})\|\widehat\xi^u - \xi^u\|_{ \partial \mathcal T_h}  - \langle f'(\bar u)\rho^u,n(\widehat \xi^u -\xi^u)\rangle_{\partial \mathcal T_h}
        \\ & 
       \leq C(\delta)\|\xi^u\|_{\mathcal{T}_h}^2 + \tfrac12\bigl\lvert f'( \bar u)\bigr\rvert\|\widehat\xi^u - \xi^u\|^2_{\partial \mathcal T_h} + Ch^{2k+1} +C \|\widehat\xi^u - \xi^u\|^2_{ \partial \mathcal T_h},
       \end{aligned}
       \end{equation*} 
\begin{equation*}
     \begin{aligned}
       f_2 &= -\bigl(\frac12\,f''_{\!u}\,(\xi^u - \rho^u)^2,\partial_x\xi^u\bigr)_{\mathcal T_h}-\langle \frac12\, f''_{\!u}\,
  ( \xi^u-\rho^u)^2,n(\widehat \xi^u -\xi^u)\rangle_{\partial \mathcal T_h} \\
       &\leq C h^{-1}|u-u_h|_{\mathcal T_h}\|\xi^u-\rho^u\|_{\mathcal{T}_h}\|\xi^u\|_{\mathcal T_h}+ C h^{-\frac{1}{2}}|u-u_h|_{\partial \mathcal T_h}\|\xi^u-\rho^u\|_{\mathcal{T}_h}\|\widehat \xi^u-\xi^u\|_{\partial \mathcal T_h} \\
       &\leq Ch^{2k+2} + C\|\xi^u\|_{ \mathcal T_h}^2+ C\|\widehat\xi^u - \xi^u\|^2_{ \partial \mathcal T_h}, 
       \end{aligned}
\end{equation*}
       \begin{equation*}
     \begin{aligned}
       f_3 = -\langle \tau_f (\widehat \xi^u - \xi^u +\rho^u),(\widehat \xi^u-\xi^u)\rangle_{\partial \mathcal T_h} = -\langle \tau_f ,(\widehat \xi^u-\xi^u)^2\rangle_{ \partial \mathcal T_h} - \langle \tau_f \rho^u,(\widehat \xi^u-\xi^u)\rangle_{\partial \mathcal T_h}.
    \end{aligned}
\end{equation*}
Hence, substituting the values of $f_i's$ into the equation \eqref{H_ij2} and employing the Young's inequality, we have 
\begin{equation}\label{aux2}
    \begin{split}
        f_1+f_2+f_3\leq - \Big\langle \tau_f - \frac{1}{2}\lvert f'(\bar u)\rvert, (\widehat \xi^u-\xi^u)^2\Big\rangle_{\partial \mathcal T_h} +C(\delta)h^{2k+1} + C\|\xi^u\|_{ \mathcal T_h}^2+ \frac{\bar C\delta}{2}\|\widehat\xi^u - \xi^u\|^2_{ \partial \mathcal T_h}.
    \end{split}
\end{equation}
The remaining parts of \eqref{auxineq1} estimated using Young's inequality and projections estimates \eqref{eq:proj-est}. Thus we have
\begin{align}\label{aux3}
    &\nonumber\bigl\langle (\rho^p- \tau_{pu}\rho^un), (\widehat \xi^u -\xi^u)n\bigl\rangle_{\partial\mathcal T_h}+\frac\beta2\langle(\rho^q-\tau_{qv}\rho^v)^2, n\rangle_{\mathcal I^-} - \frac\gamma2\langle(\rho^v-\tau_{vq}\rho^q)^2, n\rangle_{\mathcal I^+} \\& \nonumber \quad- \beta \langle(\rho^q-\tau_{qv}\rho^v), \tau_{qv}(\widehat\xi^v-\xi^v)n\rangle_{\mathcal I^-}   + \gamma \langle(\rho^v-\tau_{vq}\rho^q), \tau_{vq}(\widehat\xi^q-\xi^q)n\rangle_{\mathcal I^+}\\& \leq Ch^{2k+1} + \frac{\bar C\delta}{2} \|\widehat \xi^u -\xi^u\|_{\partial\mathcal T_h}^2 + \bar C \delta( \|\widehat \xi^q -\xi^q\|_{\mathcal I^-}^2 + \|\widehat \xi^v -\xi^v\|_{\mathcal I^+}^2).
\end{align}
Since we have $-(\tau_f -\frac{1}{2}\lvert f'(c)\rvert) \leq 0$ and $c-\bar C\delta \geq 0$ from Assumption \ref{ass:stab}, therefore, estimates \eqref{auxineq1}, \eqref{aux2}, and \eqref{aux3} together gives
\begin{align}\label{auxineq12}
&\frac12\frac{d}{dt}\|\xi^u\|_{\mathcal{T}_h}^2 \leq Ch^{2k+1} + C\|\xi^u\|_{ \mathcal T_h}^2,
\end{align}
where $\delta$ can be taken small as it comes from the application of the Young's inequality and we have ignored the positive terms from the left hand side. The generic constants $C$ and $\bar C$ are positive and independent from $h$. Applying the Gronwall's inequality, and using the fact that $\xi^u(0)=0$ (because $u_h(\cdot,0)=Pu_0$), we obtain
\[
\|\xi^u(t)\|_{\mathcal{T}_h}^2 \le C h^{2k+1}, \quad 0\le t\le T.
\]
Finally, by the triangle inequality and projection estimate \eqref{eq:proj-est}, we obtain
\[
\|u-u_h\|_{\mathcal{T}_h} \le \|\xi^u\|_{\mathcal{T}_h} + \|\rho^u\|_{\mathcal{T}_h} \le C h^{k+1/2},
\]
which completes the proof.
\end{proof}
It remains to justify an {\it a priori} assumption \eqref{ass1uu_h}.
The justification of assumption \eqref{ass1uu_h} follows via a continuity argument \cite{XuShu2007}. For $k \geq 1$ and $h$ sufficiently small, $Ch^{k+1/2} < \frac{1}{2}h$, where  $C$ is constant in \eqref{errest} can be exactly determined by using the final time $T$. Define 
\(
t^* := \sup \left\{ t : \|u(t) - u_h(t)\|_{\mathcal{T}_h} \leq h^{2}\right\}.
\)
Initial conditions satisfy $t^* > 0$ by projection properties. Assume $t^* < T$. Continuity implies 
\(
\|u(t^*) - u_h(t^*)\|_{\mathcal{T}_h} = h.
\)
Theorem \ref{thm:error} (valid for $t \leq t^*$ under the assumption) yields
\[
\|u(t^*) - u_h(t^*)\|_{\mathcal{T}_h} \leq Ch^{k+1/2} < \frac{1}{2}h.
\]
This contradicts the equality if $t^*<T$, and hence $t^* \geq T$.  The consequences in \eqref{ass1uu_h2} follow from inverse and interpolation inequalities \eqref{invesre}, and \eqref{ass1uu_h}.

\section{Time discretization}\label{sec:time}

We discretize the semi-discrete HDG formulation \eqref{eq:hdg-local}--\eqref{eq:global-transmission}
in time by an implicit $\theta$-method. Let $\Delta t>0$ be a fixed time step, $t_n:=n\Delta t$,
and denote by $(u_h^n,v_h^n,p_h^n,q_h^n)$ and by the hybrid variables
$(\widehat u_h^n,\widehat v_h^n,\widehat q_h^n)$ the approximations at time $t_n$.
We introduce the standard difference quotient
\[
\delta_t u_h^{n+1} := \frac{u_h^{n+1}-u_h^n}{\Delta t},
\]
and the intermediate (convex-combination) level $t^{n+\theta}:=t_n+\theta\Delta t$, $\theta\in[1/2,1]$,
together with
\[
w_h^{n+\theta}:=(1-\theta)w_h^n+\theta w_h^{n+1},
\qquad
\widehat w_h^{n+\theta}:=(1-\theta)\widehat w_h^n+\theta \widehat w_h^{n+1},
\]
for $w\in\{u,v,p,q\}$.

\subsection{Fully discrete HDG scheme}\label{subsec:theta_scheme}
Given $(u_h^n,v_h^n,p_h^n,q_h^n)$ and $(\widehat u_h^n,\widehat v_h^n,\widehat q_h^n)$ at time $t_n$,
we compute $(u_h^{n+1},v_h^{n+1},p_h^{n+1},q_h^{n+1})$ and
$(\widehat u_h^{n+1},\widehat v_h^{n+1},\widehat q_h^{n+1})$ by imposing the HDG equations at the intermediate
level $n+\theta$ and replacing the time derivative by $\delta_t u_h^{n+1}$.
More precisely, we seek
\[
(u_h^{n+1},q_h^{n+1},p_h^{n+1},v_h^{n+1})\in [V_h^k]^4,\quad
\widehat u_h^{n+1}\in M_h(0),\ \widehat v_h^{n+1}\in \widetilde M_h^R,\ \widehat q_h^{n+1}\in \widetilde M_h^L,
\]
such that for all test functions
$(\phi_u,\phi_q,\phi_p,\phi_v)\in [V_h^k]^4$, $\mu_v\in \widetilde M_h^R$,
$\mu_q\in \widetilde M_h^L$, and $\mu_p\in M_h(0)$,
\begin{subequations}\label{eq:global-hdg-theta}
\begin{align}
(\delta_t u_h^{n+1},\phi_u)_{\mathcal T_h}
&+ (p_h^{n+\theta},\partial_x\phi_u)_{\mathcal T_h}
- (f(u_h^{n+\theta}),\partial_x\phi_u)_{\mathcal T_h}
- \gamma(v_h^{n+\theta},\phi_u)_{\mathcal T_h}
\nonumber\\
&\qquad
- \langle \widehat{p}_h^{n+\theta}-\widehat{f(u_h^{n+\theta})},\phi_u\,n\rangle_{\partial\mathcal T_h}=0,
\label{eq:theta1}\\[2mm]
-(v_h^{n+\theta},\partial_x\phi_v)_{\mathcal T_h}
&-(u_h^{n+\theta},\phi_v)_{\mathcal T_h}
+\langle \widehat v_h^{\,n+\theta},\phi_v\,n\rangle_{\partial\mathcal T_h}=0,
\label{eq:theta2}\\[2mm]
(p_h^{n+\theta},\phi_p)_{\mathcal T_h}
&+\beta(q_h^{n+\theta},\partial_x\phi_p)_{\mathcal T_h}
-\beta\langle \widehat q_h^{\,n+\theta},\phi_p\,n\rangle_{\partial\mathcal T_h}=0,
\label{eq:theta3}\\[2mm]
(q_h^{n+\theta},\phi_q)_{\mathcal T_h}
&+(u_h^{n+\theta},\partial_x\phi_q)_{\mathcal T_h}
-\langle \widehat u_h^{\,n+\theta},\phi_q\,n\rangle_{\partial\mathcal T_h}=0.
\label{eq:theta4}
\end{align}
\end{subequations}
The transmission conditions are imposed at time level $n+\theta$:
\begin{subequations}\label{eq:theta_trans}
\begin{align}
\langle \widehat v_h^{\,n+\theta},\mu_v\,n\rangle_{\partial\mathcal T_h}
&=\langle v_R^{\,n+\theta},\mu_v\,n\rangle_{\{x_R\}},\\
\langle \widehat q_h^{\,n+\theta},\mu_q\,n\rangle_{\partial\mathcal T_h}
&=\langle q_L^{\,n+\theta},\mu_q\,n\rangle_{\{x_L\}},\\
\langle \widehat{p}_h^{n+\theta}-\widehat{f(u_h^{n+\theta})},\mu_p\,n\rangle_{\partial\mathcal T_h}
&=0,
\end{align}
\end{subequations}
for all $\mu_v \in \widetilde M_h^L$, $\mu_q \in \widetilde M_h^R$, and $\mu_v \in  M_h(0)$.
The numerical traces at $n+\theta$ are defined exactly as in the semi-discrete scheme,
but evaluated at the intermediate level in $I_i$:
\begin{subequations}\label{eq:theta_traces}
\begin{align}
\widehat{v}_h^{\,n+\theta} &= v_h^{n+\theta}+\tau_{vq}(\widehat{q}_h^{\,n+\theta}-q_h^{n+\theta})n,
&&\text{on }x_{i-1},
\\
\widehat{q}_h^{\,n+\theta} &= q_h^{n+\theta}+\tau_{qv}(\widehat{v}_h^{\,n+\theta}-v_h^{n+\theta})n,
&&\text{on }x_i,
\\
\widehat{p}_h^{\,n+\theta} &= p_h^{n+\theta}+\tau_{pu}(\widehat{u}_h^{\,n+\theta}-u_h^{n+\theta})n,
&&\text{on } \partial I_i,
\\
\widehat{f(u_h^{n+\theta})} &= f(u_h^{n+\theta})-\tau_f(\widehat{u}_h^{\,n+\theta}-u_h^{n+\theta})n,
&&\text{on }\partial I_i,
\end{align}
\end{subequations}
and the boundary conditions for each time level can be defined by \eqref{eq:physical-bc_num}. 
Since $f(u_h^{n+\theta})$ is nonlinear, \eqref{eq:global-hdg-theta}--\eqref{eq:theta_traces}
defines a nonlinear system at each time step. In practice we solve it by Newton's method or a
damped fixed-point iteration, exploiting the HDG static condensation: for given traces
$(\widehat u_h^{n+\theta},\widehat v_h^{n+\theta},\widehat q_h^{n+\theta})$, the element unknowns
$(u_h^{n+\theta},v_h^{n+\theta},p_h^{n+\theta},q_h^{n+\theta})$ are obtained locally, and the global
coupling occurs only through the transmission conditions.

We set $u_h^0:=Pu_0$. The remaining fields $(v_h^0,p_h^0,q_h^0)$ and traces can be obtained by
solving the local mixed relations $q=u_x$, $p=\beta q_x$, $v_x=u$ which are enforced in the HDG sense \eqref{eq:hdg-local} at $t=0$ with $u_h$ fixed at $u_h^0$.

\subsection{Discrete energy stability}\label{subsec:fully_discrete_stab}
Define the discrete energy
\[
E_h^n:=\frac12\|u_h^n\|_{\mathcal T_h}^2.
\]
The fully discrete stability proof follows the semi-discrete argument verbatim, except that the
time derivative term is handled by the standard $\theta$-identity below.

For any $a,b$ in an inner-product space and $\theta\in[0,1]$,
\begin{equation}\label{theta_identity}
    (b-a,(1-\theta)a+\theta b)=\frac12\Big(\|b\|^2-\|a\|^2+(2\theta-1)\|b-a\|^2\Big).
\end{equation}

\begin{theorem}[Fully discrete stability]\label{thm:fd_stability}
Assume homogeneous boundary data and Assumption~\ref{ass:stab} on the stabilization parameters.
Then the fully discrete scheme \eqref{eq:global-hdg-theta}--\eqref{eq:theta_traces} satisfies the
energy inequality
\[
E_h^{n+1}-E_h^n + \frac{2\theta-1}{2}\|u_h^{n+1}-u_h^n\|_{\mathcal T_h}^2
+ \Delta t\,\mathcal D_h^{n+\theta}\ \le\ 0,
\qquad \forall n\ge 0,
\]
where $\mathcal D_h^{n+\theta}\ge 0$ is the same (trace-based) dissipation functional as in the
semi-discrete identity \eqref{identity_31}, evaluated at time level $n+\theta$.
\end{theorem}

\begin{proof}
Choosing the test functions:
\[
\phi_u=u_h^{n+\theta},\qquad \phi_v=-\gamma v_h^{n+\theta},\qquad \phi_p=-q_h^{n+\theta},\qquad
\phi_q=p_h^{n+\theta},
\]
in \eqref{eq:global-hdg-theta}, adding \eqref{eq:theta1}--\eqref{eq:theta4}, and applying the similar algebraic manipulations along with multiple integration by parts and transmission conditions give exactly the semi-discrete energy balance \eqref{identity_31} at level $n+\theta$,  except for the time term $(\delta_t u_h^{n+1},u_h^{n+\theta})_{\mathcal T_h}$:
\begin{equation*}
    (\delta_t u_h^{n+1},u_h^{n+\theta})_{\mathcal T_h} +\mathcal D_h^{n+\theta} = 0.
\end{equation*}
Furthermore, applying identity \eqref{theta_identity} elementwise yields
\[
(\delta_t u_h^{n+1},u_h^{n+\theta})_{\mathcal T_h}
=\frac{1}{2\Delta t}\Big(\|u_h^{n+1}\|_{\mathcal T_h}^2-\|u_h^n\|_{\mathcal T_h}^2
+(2\theta-1)\|u_h^{n+1}-u_h^n\|_{\mathcal T_h}^2\Big).
\]
A dissipation functional
$\mathcal D_h^{n+\theta}$ controlled by Assumption~\ref{ass:stab} and the boundary conditions. Since $\theta\ge 1/2$, the extra
term $\frac{2\theta-1}{2}\|u_h^{n+1}-u_h^n\|^2$ is nonnegative and the conclusion follows.
\end{proof}

\begin{remark}[Role of $\theta$]
The Crank--Nicolson choice $\theta=\tfrac12$ removes the additional time-dissipation term
$\frac{2\theta-1}{2}\|u_h^{n+1}-u_h^n\|^2$, making the method the least dissipative in time.
For $\theta=1$ (backward Euler), the method adds numerical dissipation in time.
\end{remark}

\subsection{Fully discrete error analysis}\label{subsec:fd_error}

We present the fully discrete error analysis for the $\theta$-scheme
\eqref{eq:global-hdg-theta}--\eqref{eq:theta_traces}. Throughout we assume the stabilization
parameters satisfy Assumption~\ref{ass:stab}. For any sufficiently smooth scalar function
$g(t)$ define the convex combination of its endpoint values
\[
g^{n+\theta,*}:=(1-\theta)g(t_n)+\theta g(t_{n+1}),
\]
and the intermediate-time mismatch
\begin{equation}\label{eq:mu_def_fd}
\mu_g^{n+\theta}:=g^{n+\theta,*}-g(t^{n+\theta}).
\end{equation}
 We also define the time discretization error for $u$ at intermediate step:
\begin{equation}\label{eq:eta_def_fd}
\eta_u^{n+\theta}:=\frac{u(t_{n+1})-u(t_n)}{\Delta t}-u_t(t^{n+\theta})=: \delta_t u(t_{n+1}) - u_t(t^{n+\theta}).
\end{equation}
For the nonlinear flux we set
\begin{equation}\label{eq:Mf_def_fd}
\mathcal M_f^{n+\theta}:= f\big(u(t^{n+\theta})\big)-f\big(u^{n+\theta,*}\big).
\end{equation}
Write $t_*:=t^{n+\theta}$. Using Taylor expansions around $t_*$,
\[
u(t_{n+1}) = u(t_*) + (1-\theta)\Delta t\,u_t(t_*) + \frac{(1-\theta)^2\Delta t^2}{2}\,u_{tt}(\xi_1),
\]
\[
u(t_{n}) = u(t_*) - \theta\Delta t\,u_t(t_*) + \frac{\theta^2\Delta t^2}{2}\,u_{tt}(\xi_0),
\]
for some $\xi_0,\xi_1\in(t_n,t_{n+1})$. Subtracting and dividing by $\Delta t$ gives the explicit
structure
\begin{equation}\label{eq:eta_expansion}
\eta_u^{n+\theta}
=\Big(\tfrac12-\theta\Big)\Delta t\,u_{tt}(t_*) + \mathcal{O}(\Delta t^2),
\end{equation}
where the $\mathcal{O}(\Delta t^2)$ term involves $u_{ttt}$. Therefore, if $\theta\neq \tfrac12$ (in particular $\theta=1$ backward Euler), the leading term in
\eqref{eq:eta_expansion} is proportional to $\Delta t\,u_{tt}$, hence the method is first order in time,
and if $\theta=\tfrac12$ (Crank--Nicolson)
so $\eta_u^{n+\frac12}=\mathcal{O}(\Delta t^2)$ and the method is second order in time.
Thus a simple applications of Taylor's theorem proves the following lemma on time bound estimates.
\begin{lemma}[Time remainder bounds]\label{lem:time_bounds_fd}
Assume $u_{tt}\in L^2(t_n,t_{n+1};L^2(\Omega))$. Then for any $\theta\in[0,1]$,
\begin{align}
\|\eta_u^{n+\theta}\|_{L^2(\Omega)}^2
&\le C\,\Delta t \int_{t_n}^{t_{n+1}}\|u_{tt}(t)\|_{L^2(\Omega)}^2\,dt, \label{eq:eta_bound_fd}\\
\|\mu_u^{n+\theta}\|_{L^2(\Omega)}^2
&\le C\,\Delta t^3 \int_{t_n}^{t_{n+1}}\|u_{tt}(t)\|_{L^2(\Omega)}^2\,dt, \label{eq:mu_bound_fd}
\end{align}
and similarly for $\mu_v^{n+\theta}$, $\mu_p^{n+\theta}$, $\mu_q^{n+\theta}$ under the corresponding time regularity.
Moreover, if $\theta=\tfrac12$ and $u_{ttt}\in L^2(t_n,t_{n+1};L^2(\Omega))$, then
\begin{equation}\label{eq:eta_CN_bound_fd}
\|\eta_u^{n+\frac12}\|_{L^2(\Omega)}^2
\le C\,\Delta t^3 \int_{t_n}^{t_{n+1}}\|u_{ttt}(t)\|_{L^2(\Omega)}^2\,dt.
\end{equation}
If $f$ is locally Lipschitz and $u$ is bounded, then
\begin{equation}\label{eq:Mf_bound_fd}
\|\mathcal M_f^{n+\theta}\|_{L^2(\Omega)} \le C\,\|\mu_u^{n+\theta}\|_{L^2(\Omega)}.
\end{equation}
\end{lemma}

Define endpoint projection errors and discrete errors (as in the semi-discrete analysis):
\[
e_\omega^n:=\omega(t_n)-\omega_h^n,\qquad
e_\omega^n=\xi_\omega^n-\rho_\omega^n,
\qquad
\xi_\omega^n:=P\omega(t_n)-\omega_h^n,\qquad
\rho_\omega^n:=P\omega(t_n)-\omega(t_n),
\]
for $\omega\in\{u,v,p,q\}$. At the intermediate level we use the convex combinations
\[
\xi_\omega^{n+\theta}:=(1-\theta)\xi_\omega^n+\theta\xi_\omega^{n+1},\qquad
\rho_\omega^{n+\theta}:=(1-\theta)\rho_\omega^n+\theta\rho_\omega^{n+1}.
\]
On the skeleton, define the intermediate-level trace errors using the same convex-combination
reference values:
\begin{equation}\label{eq:hat_xi_def_fd}
\widehat\xi_u^{n+\theta}:=u^{n+\theta,*}-\widehat u_h^{\,n+\theta},\quad
\widehat\xi_v^{n+\theta}:=v^{n+\theta,*}-\widehat v_h^{\,n+\theta},\quad
\widehat\xi_q^{n+\theta}:=q^{n+\theta,*}-\widehat q_h^{\,n+\theta},\quad
\widehat\xi_p^{n+\theta}:=p^{n+\theta,*}-\widehat p_h^{\,n+\theta}.
\end{equation}
(Here $\widehat v_h^{\,n+\theta},\widehat q_h^{\,n+\theta},\widehat p_h^{\,n+\theta}$ are the numerical traces defined by
\eqref{eq:theta_traces}.) Exactly as in \eqref{define_hatxi}--\eqref{num_p_tracexi} of the semi-discrete analysis, the trace
definitions imply the identities
\begin{align}
\widehat\xi_v^{n+\theta}
&=\xi_v^{n+\theta}+\tau_{vq}\big(\widehat\xi_q^{n+\theta}-\xi_q^{n+\theta}\big)n
-\Big(\rho_v^{n+\theta}-\tau_{vq}\rho_q^{n+\theta}n\Big),
&&\text{at }x_{i-1},\label{eq:hatxi_v_id_fd}\\
\widehat\xi_q^{n+\theta}
&=\xi_q^{n+\theta}+\tau_{qv}\big(\widehat\xi_v^{n+\theta}-\xi_v^{n+\theta}\big)n
-\Big(\rho_q^{n+\theta}-\tau_{qv}\rho_v^{n+\theta}n\Big),
&&\text{at }x_i,\label{eq:hatxi_q_id_fd}\\
\widehat\xi_p^{n+\theta}
&=\xi_p^{n+\theta}+\tau_{pu}\big(\widehat\xi_u^{n+\theta}-\xi_u^{n+\theta}\big)n
-\Big(\rho_p^{n+\theta}-\tau_{pu}\rho_u^{n+\theta}n\Big),
&&\text{on }\partial I_i.\label{eq:hatxi_p_id_fd}
\end{align}

Let $\phi_u,\phi_v,\phi_p,\phi_q\in V_h^k$ be arbitrary.
Subtracting the fully discrete scheme
\eqref{eq:global-hdg-theta} (evaluated at $t^{n+\theta}$) from the exact weak identities at time $t^{n+\theta}$
and rewriting all exact quantities in terms of $\,(\cdot)^{n+\theta,*}$ plus the remainders
\eqref{eq:mu_def_fd}--\eqref{eq:Mf_def_fd} yields the following system:

\begin{subequations}\label{eq:fd_error_system}
\begin{align}
(\delta_t \xi_u^{n+1},\phi_u)_{\mathcal T_h}
&+(\xi_p^{n+\theta},\partial_x\phi_u)_{\mathcal T_h}
-\big(f(u^{n+\theta,*})-f(u_h^{n+\theta}),\partial_x\phi_u\big)_{\mathcal T_h}
-\gamma(\xi_v^{n+\theta},\phi_u)_{\mathcal T_h}
\nonumber\\
&\quad
-\langle \widehat\xi_p^{n+\theta},\phi_u\,n\rangle_{\partial\mathcal T_h}
+\big\langle f(u^{n+\theta,*})-\widehat{f(u_h^{n+\theta})},\phi_u\,n\big\rangle_{\partial\mathcal T_h}
= \mathcal R_u^{n+\theta}(\phi_u),
\label{eq:fd_err_u}\\[1mm]
-(\xi_v^{n+\theta},\partial_x\phi_v)_{\mathcal T_h}
&-(\xi_u^{n+\theta},\phi_v)_{\mathcal T_h}
+\langle \widehat\xi_v^{n+\theta},\phi_v\,n\rangle_{\partial\mathcal T_h}
= \mathcal R_v^{n+\theta}(\phi_v),
\label{eq:fd_err_v}\\[1mm]
(\xi_p^{n+\theta},\phi_p)_{\mathcal T_h}
&+\beta(\xi_q^{n+\theta},\partial_x\phi_p)_{\mathcal T_h}
-\beta\langle \widehat\xi_q^{n+\theta},\phi_p\,n\rangle_{\partial\mathcal T_h}
= \mathcal R_p^{n+\theta}(\phi_p),
\label{eq:fd_err_p}\\[1mm]
(\xi_q^{n+\theta},\phi_q)_{\mathcal T_h}
&+(\xi_u^{n+\theta},\partial_x\phi_q)_{\mathcal T_h}
-\langle \widehat\xi_u^{n+\theta},\phi_q\,n\rangle_{\partial\mathcal T_h}
= \mathcal R_q^{n+\theta}(\phi_q),
\label{eq:fd_err_q}
\end{align}
\end{subequations}
where the remainder functionals are fully explicit:
\begin{subequations}\label{eq:fd_remainders}
\begin{align}
\mathcal R_u^{n+\theta}(\phi_u)
&:= (\eta_u^{n+\theta},\phi_u)_{\mathcal T_h}
+(\mu_p^{n+\theta},\partial_x\phi_u)_{\mathcal T_h}
+(\mathcal M_f^{n+\theta},\partial_x\phi_u)_{\mathcal T_h}
-\gamma(\mu_v^{n+\theta},\phi_u)_{\mathcal T_h}
\nonumber\\
&\qquad
-\big\langle \mu_p^{n+\theta}+\mathcal M_f^{n+\theta},\phi_u\,n\big\rangle_{\partial\mathcal T_h},
\label{eq:R_u_fd}\\
\mathcal R_v^{n+\theta}(\phi_v)
&:= -(\mu_v^{n+\theta},\partial_x\phi_v)_{\mathcal T_h}
-(\mu_u^{n+\theta},\phi_v)_{\mathcal T_h}
+\langle \mu_v^{n+\theta},\phi_v\,n\rangle_{\partial\mathcal T_h},
\label{eq:R_v_fd}\\
\mathcal R_p^{n+\theta}(\phi_p)
&:= -(\mu_p^{n+\theta},\phi_p)_{\mathcal T_h}
-\beta(\mu_q^{n+\theta},\partial_x\phi_p)_{\mathcal T_h}
+\beta\langle \mu_q^{n+\theta},\phi_p\,n\rangle_{\partial\mathcal T_h},
\label{eq:R_p_fd}\\
\mathcal R_q^{n+\theta}(\phi_q)
&:= -(\mu_q^{n+\theta},\phi_q)_{\mathcal T_h}
-(\mu_u^{n+\theta},\partial_x\phi_q)_{\mathcal T_h}
+\langle \mu_u^{n+\theta},\phi_q\,n\rangle_{\partial\mathcal T_h}.
\label{eq:R_q_fd}
\end{align}
\end{subequations}
where we have used
\begin{align*}
(u_t(t^{n+\theta}),\phi_u)_{\mathcal T_h}-(\delta_t u_h^{n+1},\phi_u)_{\mathcal T_h}
&=(\delta_t u(t_{n+1}),\phi_u)_{\mathcal T_h}-(\eta_u^{n+\theta},\phi_u)_{\mathcal T_h}-(\delta_t u_h^{n+1},\phi_u)_{\mathcal T_h}\\
& =(\delta_t \xi_u^{n+1},\phi_u)_{\mathcal T_h} -(\eta_u^{n+\theta},\phi_u)_{\mathcal T_h},
\end{align*}
since $P$ is the $L^2$-projection onto $V_h^k$, for any $\phi_u\in V_h^k$, we have
$(\delta_t Pu(t_{n+1}),\phi_u)=(\delta_t u(t_{n+1}),\phi_u)$, similarly for other variable.
Under homogeneous boundary data, the transmission conditions yield the same relations as in the semi-discrete case:
\begin{subequations}\label{eq:fd_trans_error}
\begin{align}
\langle \widehat\xi_v^{n+\theta},\mu_v\,n\rangle_{\partial\mathcal T_h} &= 0,
&&\forall \mu_v\in \widetilde M_h^R,\\
\langle \widehat\xi_q^{n+\theta},\mu_q\,n\rangle_{\partial\mathcal T_h} &= 0,
&&\forall \mu_q\in \widetilde M_h^L,\\
\big\langle \widehat\xi_p^{n+\theta} - \big(f(u^{n+\theta,*})-\widehat{f(u_h^{n+\theta})}\big),\mu_p\,n\big\rangle_{\partial\mathcal T_h} &= 0,
&&\forall \mu_p\in M_h(0).
\end{align}
\end{subequations}

Choose the similar test functions as in the semi-discrete error proof of Theorem \ref{thm:error}, now at level $n+\theta$:
\[
\phi_u=\xi_u^{n+\theta},\qquad
\phi_v=-\gamma\xi_v^{n+\theta},\qquad
\phi_p=-\xi_q^{n+\theta},\qquad
\phi_q=\xi_p^{n+\theta}.
\]
Add \eqref{eq:fd_err_u}--\eqref{eq:fd_err_q} and use the same algebraic manipulations, integration-by-parts,
and transmission-condition arguments as in the semi-discrete case to obtain the discrete error energy identity,
with the only new step being the time term:
\[
(\delta_t \xi_u^{n+1},\xi_u^{n+\theta})_{\mathcal T_h}
=\frac{1}{2\Delta t}\Big(\|\xi_u^{n+1}\|_{\mathcal T_h}^2-\|\xi_u^{n}\|_{\mathcal T_h}^2
+(2\theta-1)\|\xi_u^{n+1}-\xi_u^{n}\|_{\mathcal T_h}^2\Big)
\quad\text{(by \eqref{theta_identity}).}
\]
The right-hand side contributions are exactly the remainder functionals
\eqref{eq:fd_remainders} tested with the chosen test functions and are bounded using
Cauchy--Schwarz and Young's inequality together with Lemma~\ref{lem:time_bounds_fd} and the spatial projection estimates
from the semi-discrete analysis. Summing over $n$ and applying a discrete Gronwall inequality  and combining with the spatial estimate from the semi-discrete analysis (order $h^{k+\frac12}$ ) finally gives
\[
\max_{0\le n\le N}\|u(t_n)-u_h^n\|_{L^2(\Omega)} \le
C\Big(h^{k+\frac12}+\Delta t^{\,r}\Big),
\qquad
r=
\begin{cases}
1, & \theta\neq \tfrac12\ \text{(e.g. backward Euler }\theta=1),\\
2, & \theta=\tfrac12\ \text{(Crank--Nicolson)},
\end{cases}
\]
under the corresponding time regularity assumed in Lemma \ref{lem:time_bounds_fd}.

\section{Numerical examples}\label{sec:numerics}
We now present two numerical examples to verify the convergence of the proposed HDG method. In following examples, the stabilization parameters are chosen according to Assumption \ref{ass:stab}. In particular, we use the (constant) stabilization parameters
\[
\tau_{pu}=2,\qquad 
\tau_{vq}=0.9\sqrt{\beta/\gamma},\qquad 
\tau_{qv}=0.9\sqrt{\gamma/\beta},\qquad 
\tau_f=2.
\]

\subsection*{Example 6.1. Smooth manufactured solution (Accuracy test)}

We assess the spatial accuracy of the proposed HDG scheme using a smooth manufactured solution on
$\Omega=(0,2\pi)$ with parameters $\beta=0.5$ and $\gamma=1$.
We prescribe the exact solution
\[
u(x,t)=e^{-t}\,\sin(x),
\qquad t\in[0,T],\quad T=0.5.
\]
The source term $g(x,t)$ is chosen so that $(u,v)$ satisfies the Ostrovsky system \eqref{eq:ostrovsky}.
We impose boundary data consistent with the manufactured solution.

For time integration we use the $\theta$-scheme with $\theta=\tfrac12$ (Crank--Nicolson) and $\theta=1$ (backward Euler implicit) with  fixed time steps
$\Delta t=0.001$ and $\Delta t=0.1h^{k+1}$ respectively, chosen sufficiently small so that the temporal error does not dominate the spatial error on the meshes considered. For each polynomial degree $k\in\{1,2,3\}$ we compute the numerical solution at $T=0.5$.
Although Theorem~\ref{thm:error} guarantees an $O(h^{k+1/2})$ bound in $L^2(\Omega)$, the observed rates in
Table~\ref{tab:smooth_conv_merged} are closer to $k+1$.
This behavior is consistent with the commonly observed superconvergent behavior of hybridized methods in practice.

\begin{table}[t]
\centering
\renewcommand{\arraystretch}{1.15}
\setlength{\tabcolsep}{6pt}
\begin{tabular}{|c|cc|cc|cc|}
\hline
 & \multicolumn{2}{c|}{$k=1$} & \multicolumn{2}{c|}{$k=2$} & \multicolumn{2}{c|}{$k=3$} \\
\cline{2-7}
$N_e$ & error & rate & error & rate & error & rate \\
\hline
\hline
\multicolumn{7}{|c|}{\textbf{$\theta=\tfrac12$ (Crank--Nicolson), $\Delta t= 0.001$}}\\
\hline
  2  & $8.022\times10^{-1}$ & --    & $6.661\times10^{-2}$ & --    & $3.427\times10^{-2}$ & --    \\
  4  & $1.355\times10^{-1}$ & 2.565 & $1.671\times10^{-2}$ & 1.995 & $1.406\times10^{-3}$ & 4.607 \\
  8  & $3.013\times10^{-2}$ & 2.169 & $1.765\times10^{-3}$ & 3.243 & $8.572\times10^{-5}$ & 4.036 \\
 16  & $6.907\times10^{-3}$ & 2.125 & $2.152\times10^{-4}$ & 3.036 & $5.340\times10^{-6}$ & 4.005 \\
 32  & $1.670\times10^{-3}$ & 2.048 & $2.680\times10^{-5}$ & 3.006 & $3.658\times10^{-7}$ & 3.868 \\
\hline
\multicolumn{7}{|c|}{\textbf{$\theta=1$ (backward Euler), $\Delta t=0.1\,h^{k+1}$}}\\
\hline
  2  & $7.375\times10^{-1}$ & --    & $5.471\times10^{-2}$ & --    & $1.083\times10^{-1}$ & --    \\
  4  & $1.293\times10^{-1}$ & 2.512 & $1.784\times10^{-2}$ & 1.617 & $1.740\times10^{-2}$ & 2.638 \\
  8  & $2.920\times10^{-2}$ & 2.147 & $2.270\times10^{-3}$ & 2.975 & $1.237\times10^{-3}$ & 3.814 \\
 16  & $6.788\times10^{-3}$ & 2.105 & $2.895\times10^{-4}$ & 2.971 & $7.787\times10^{-5}$ & 3.990 \\
\hline
\end{tabular}
\caption{Convergence test at $T=0.5$ with $(\alpha,\beta,\gamma=(1,0.5,1)$.
We report the $L^2$ error in $u$ and the observed rate for polynomial degrees $k=1,2,3$.}
\label{tab:smooth_conv_merged}
\end{table}

\subsection*{Example 6.2. Periodic solitary-wave propagation (numerical profile)}

We test whether the proposed HDG scheme \eqref{eq:hdg-local}--\eqref{eq:global-transmission}
propagates a localized solitary-wave profile on a large periodic domain without spurious
distortion.

We consider the Ostrovsky equation \eqref{eq:ostrovsky} in the mixed first-order form
\eqref{eq:first-order} on the periodic interval $\Omega=(0,L)$ with $L\gg 1$.
Since $v_x=u$ determines $v$ only up to an additive constant, we fix the gauge by
prescribing the mean-zero compatibility
\begin{equation}\label{eq:mean_zero_soliton}
\int_0^L u(x,t)\,dx=0,
\end{equation}
so that the periodic anti--derivative (generalized inverse) $\partial_x^{-1}$ is well-defined on
mean-zero functions and $v$ is uniquely determined once a reference value (e.g.\ $v(0,t)=0$) is chosen.

A solitary traveling wave is a coherent structure of the form
\begin{equation}\label{eq:tw_ansatz_soliton}
u(x,t)=U(\xi),\qquad \xi=x-c_w t,
\end{equation}
where $U$ is localized on $\mathbb{R}$ and, on a sufficiently large periodic box, is exponentially small
near $x=0$ and $x=L$.
From \eqref{eq:first-order} we have $q=U'$, $p=\beta U''$, and $v_x=u$ implies
$v(\xi)=\partial_\xi^{-1}U(\xi)$, interpreted as the mean--zero periodic anti--derivative
consistent with \eqref{eq:mean_zero_soliton}.
Substituting \eqref{eq:tw_ansatz_soliton} into the last equation of \eqref{eq:first-order} yields
\begin{equation}\label{eq:tw_before_integrate_soliton}
-c_w U' - \beta U''' + \bigl(f(U)\bigr)' - \gamma\, \partial_\xi^{-1}U = 0 .
\end{equation}
Integrating once in $\xi$ and using the mean-zero periodic convention (so the integration constant is zero)
gives the stationary profile equation
\begin{equation}\label{eq:tw_profile_equation_soliton}
-c_w U - \beta U'' + f(U) - \gamma\, \partial_\xi^{-2}U = 0,
\qquad \widehat{U}(0)=0,
\end{equation}
where $\partial_\xi^{-2}$ denotes the periodic inverse of $\partial_\xi^2$ acting on mean-zero functions.

Let $\widehat{U}(\kappa)$ denote the Fourier coefficients on $(0,L)$ with wave numbers
$\kappa=\frac{2\pi}{L}m$, $m\in\mathbb{Z}$.
Taking Fourier transforms in \eqref{eq:tw_profile_equation_soliton} yields, for $\kappa\neq 0$,
\begin{equation}\label{eq:fourier_profile_soliton}
\Bigl(\beta \kappa^2 - c_w + \frac{\gamma}{\kappa^2}\Bigr)\widehat{U}(\kappa)
+ \widehat{f(U)}(\kappa)=0,\qquad \widehat{U}(0)=0.
\end{equation}
Thus the nonzero modes satisfy the nonlinear fixed-point map
\begin{equation}\label{eq:fixed_point_map_soliton}
\widehat{U}(\kappa)=
-\frac{\widehat{f(U)}(\kappa)}{\beta \kappa^2 - c_w + \gamma/\kappa^2},
\qquad \kappa\neq 0,
\quad \widehat{U}(0)=0.
\end{equation}
In our computations we take $f(u)=\frac{\alpha}{2}u^2$ and construct a numerical profile by a
stabilized (Petviashvili-type) fixed-point iteration \cite{duran2018numerical}.
Given an iterate $U^{(n)}$, we compute $F^{(n)}=\frac{\alpha}{2}(U^{(n)})^2$, evaluate its Fourier
coefficients $\widehat{F^{(n)}}$ by FFT, and update for $\kappa\neq 0$ by
\begin{equation}\label{eq:petviashvili_update_soliton}
\widehat{U}^{(n+1)}(\kappa)
=
-\,M_n^{\,p}\,
\frac{\widehat{F^{(n)}}(\kappa)}{\beta \kappa^2 - c_w + \gamma/\kappa^2},
\qquad
\widehat{U}^{(n+1)}(0)=0,
\end{equation}
where $M_n$ is an amplitude-correction factor and $p>0$ is an exponent suitable for quadratic nonlinearities. 
We apply mild relaxation
$U^{(n+1)}\leftarrow (1-\omega)U^{(n)}+\omega U^{(n+1)}$ with $\omega\in(0,1]$, and enforce the
mean-zero constraint by subtracting the spatial average after each iterate. For the quadratic nonlinearity we take the Petviashvili exponent $p=2$ and use relaxation $\omega=0.8$.
We iterate until the residual of \eqref{eq:fourier_profile_soliton} is below a prescribed tolerance.
On a large box, the resulting periodic profile satisfies $|U(0)|\approx |U(L)|\ll 1$, i.e.\ the tails are
negligible at the endpoints.

We take the computed profile as initial data $u_0(x)=U(x-x_0)$ (shifted so the peak is away from the
boundaries) and initialize $q_0=u_{0x}$, $p_0=\beta u_{0xx}$, and $v_0=\partial_x^{-1}u_0$ using the same
periodic mean-zero convention \eqref{eq:mean_zero_soliton}. 

To assess shape preservation, we compare the HDG solution $u_h(x,t)$ with the traveling--wave reference
obtained by periodic translation of the numerical profile:
\begin{equation}\label{eq:travelling_reference}
u_{\mathrm{ref}}(x,t) := U(x-c_w t).
\end{equation}
This $u_{\mathrm{ref}}$ is not an exact solution of the full PDE on a finite periodic box, but it is the
appropriate benchmark for short-time propagation.

In this experiment we take $\alpha=2$, $\beta=1$, $\gamma=\tfrac14$, $L=80$ (periodic),
and wave speed $c_w=-0.75$ (which lies in the regime $c_w<2\sqrt{\beta\gamma}$ considered in the reference
setup). We compute the numerical profile on a Fourier grid of $K=512$ points and then evolve the HDG scheme on
a mesh of $N_e$ elements with polynomial degree $k$ and a Crank--Nicolson time discretization
($\theta=\tfrac12$).

Figure~\ref{fig:ov_soliton} shows snapshots of $u_h(x,t)$ at times $t=0,5,10,15,20$, together with the
translated reference profile $u_{\mathrm{ref}}(x,T)$ at the final time.
We observe that the wave translates to the right/left according to the sign of $c_w$ and remains localized,
with only mild dispersive radiation, indicating that the HDG discretization captures coherent travelling-wave
dynamics of the Ostrovsky model.

\begin{figure}[t]
  \centering
  \includegraphics[width=0.92\linewidth]{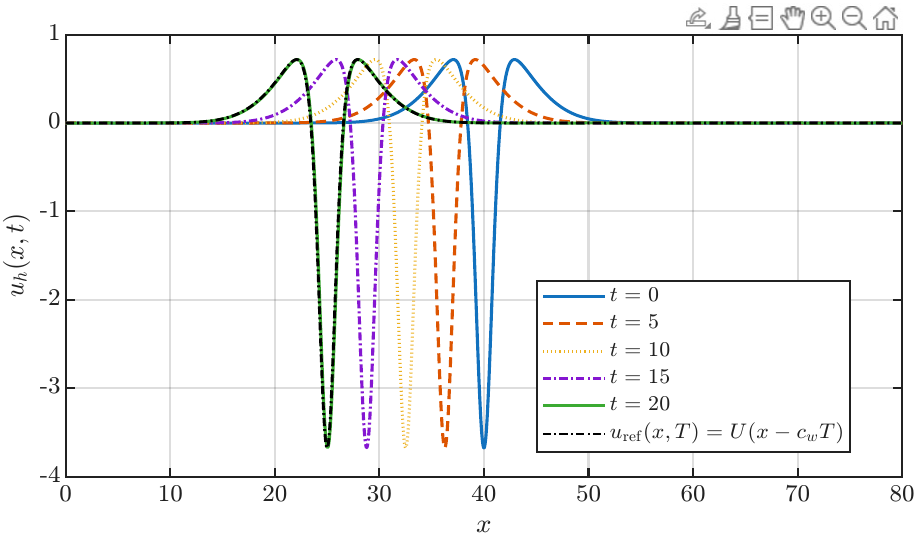}
  \caption{Solitary-wave propagation with $N_e=256$, $k=2$, $\Delta t=0.05$, $T=20$.
Shown are $u_h(\cdot,t)$ at $t=0,5,10,15,20$ and the translated reference $u_{\rm ref}(\cdot,T)=U(\cdot-c_wT)$.
}
  \label{fig:ov_soliton}
\end{figure}

Solitary waves in the Ostrovsky model \eqref{eq:ostrovsky} represent long, weakly nonlinear dispersive waves in a rotating medium
(e.g.\ internal waves under the Coriolis effect). They arise from a balance of nonlinearity and dispersion
modified by rotation, producing a coherent structure that propagates over long times with a nearly unchanged
shape. The present experiment verifies that the HDG discretization reproduces this qualitative behaviour and
provides a practical benchmark for
tracking numerical dispersion and dissipation \cite{coclite2020solutions,duran2018numerical,ostrovsky1978nonlinear}.


\subsection*{Example 6.3. Peakon solution and the singular limit}

We illustrate a non-smooth benchmark and the singular regime $\beta\to 0$, where the Ostrovsky equation \eqref{eq:ostrovsky}
converges to the Ostrovsky--Hunter (OH) equation; see, e.g., \cite{coclite2014convergence}.

Following  \cite[Example 4.3]{zhang2020discontinuous}, we take the ``peakon''
\begin{equation}\label{eq:peakon_u0}
u_0(x)=
\begin{cases}
\displaystyle \frac16\Big(x-\frac12\Big)^2+\frac16\Big(x-\frac12\Big)+\frac1{36}, & x\in\Big[0,\frac12\Big],\\[2mm]
\displaystyle \frac16\Big(x-\frac12\Big)^2-\frac16\Big(x-\frac12\Big)+\frac1{36}, & x\in\Big[\frac12,1\Big],
\end{cases}
\qquad u_0(x+1)=u_0(x).
\end{equation}
For the OH case ($\beta=0$), this yields the explicit traveling-wave solution
\begin{equation}\label{eq:peakon_exact}
u(x,t)=u_0\!\left(x-\frac{t}{36}\right),
\end{equation}
so the profile translates at speed $c=\tfrac1{36}$ and returns to $u_0$ at times $T\in 36\mathbb{N}$.
To initialize the mixed variables in \eqref{eq:first-order} we set
$q(\cdot,0)=\partial_x u_0$, $v(\cdot,0)=\int_0^x u_0(s)\,ds$ (hence $v(0,0)=0$ and this is the periodic mean-zero antiderivative gauge), and $p(\cdot,0)=\beta\,\partial_x q(\cdot,0)$.

The Ostrovsky equation models weakly nonlinear long waves in a rotating fluid; the term involving $\gamma$
encodes the large-scale restoring (Coriolis) effect, while $\beta$ controls dispersive regularization \cite{coclite2014convergence,coclite2015dispersive}.
The peakon-type profile \eqref{eq:peakon_u0} represents a coherent structure with a sharp corner
(discontinuous slope), making it a demanding benchmark for high-order methods.

We fix $\alpha=1$ and $\gamma=1$
and compute periodic HDG solutions for a sequence $\beta\in\{0,\beta_1,\beta_2,\dots\}$ with $\beta_j\to 0$.
For $\beta=0$ we validate against the explicit wave \eqref{eq:peakon_exact}. For $\beta>0$ we quantify the
singular-limit behavior by comparing the OV solution to the OH reference profile at the same final time $T=2$. We used polynomial degree $k=2$, elements $N_e=32$, time step $\Delta t=0.005$, $\theta=\frac{1}{2}$, and $\beta\in\{0, 10^{-4}, 10^{-5},10^{-6}\}$.

Figure~\ref{fig:peakon_ov_to_oh} shows the initial profile $u_0$, the OH reference profile at time $T$,
and the HDG approximations for several values of $\beta$.
As $\beta$ decreases, the numerical Ostrovsky curves approach the OH reference, indicating that the proposed HDG method captures the singular limit while accurately transporting a corner-type coherent structure.

\begin{figure}[t]
\centering
\includegraphics[width=0.85\linewidth]{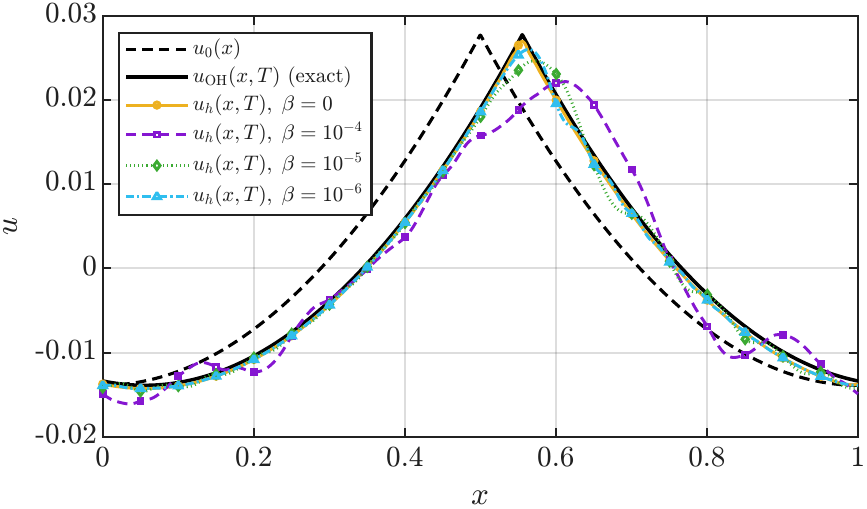}
\caption{Peakon solution and the limit $\beta\to0$.
Shown are the initial condition $u_0$ from \eqref{eq:peakon_u0}, the OH reference profile $u_{\mathrm{OH}}(\cdot,T)$
given by \eqref{eq:peakon_exact}, and periodic HDG solutions $u_h(\cdot,T)$
for several values of $\beta$.}
\label{fig:peakon_ov_to_oh}
\end{figure}

\section{Conclusion}\label{sec:conclusion}
We developed a HDG method for the Ostrovsky equation on a bounded interval by localizing the nonlocal term through the auxiliary variable $v$ defined by $v_x=u$ with a fixed boundary constraint to ensure uniqueness.
The resulting formulation admits elementwise elimination of interior unknowns and a global system posed only in terms of numerical traces.
For $\beta\neq 0$ and $\gamma>0$ (under the stated stabilization conditions), we established an $L^2$-stability estimate and proved an {\it a priori} $L^2$-error bound for the primary variable $u$.
Numerical experiments confirm the predicted convergence for smooth solutions and demonstrate that the method remains accurate and stable in challenging regimes, including near non-smooth peaked profiles for reduced OH model whenever $\beta \to 0$.
 The present HDG design extends naturally to KP-type models \cite{KadomtsevPetviashvili1970} in two space dimensions in which the transverse term appears as $\partial_x^{-1}u_{yy}$:
one introduces auxiliary variables for $y$-derivatives (e.g.\ $w=u_y$, $r=w_y$) and enforces the inverse-$x$ operator through a mixed constraint $v_x=r$ together with an appropriate gauge condition.

\section*{Declarations}

\subsection*{Data availability}
All numerical experiments can be reproduced from the implementation used in this work. The data supporting the findings of this study are available in the accompanying public repository \cite{rupp2026hdgostrovsky}. 

\subsection*{Competing interests}
The authors declare that they have no competing interests.

\subsection*{Funding}
This work has been supported by the Deutsche Forschungsgemeinschaft (DFG, German
Research Foundation) -- 577175348.

\bibliographystyle{abbrv}
\bibliography{main_ov}

\end{document}